\documentclass[12pt]{article}
\usepackage{amsthm}
\usepackage{graphicx}
\usepackage{caption}
\usepackage{subcaption}
\usepackage{amssymb}
\usepackage{amsmath}
\usepackage{xfrac}
\usepackage{color}
\usepackage{hyperref}
\usepackage{commath}
\usepackage{mathrsfs}
\usepackage[normalem]{ulem}
\usepackage{cleveref}

\usepackage{todonotes}
\setuptodonotes{inline}

\theoremstyle{definition}
\newtheorem*{definition}{Definition}

\newtheorem{theorem}{Theorem}[section]
\newtheorem{lemma}[theorem]{Lemma}
\newtheorem{claim}[theorem]{Claim}

\newtheorem{corollary}[theorem]{Corollary}

\newtheorem{observation}[theorem]{\textbf{Observation}}
\newtheorem*{remark}{Remark}

\newcommand\cJ{\mathcal J}

\newcommand\cP{\mathcal P}

\newcommand\cY{{\mathcal Y}}

\DeclareMathOperator{\bo}{O}
\DeclareMathOperator{\wpn}{wpn}

\DeclareMathOperator{\Bell}{Bell}

\title{Typical $T$-free graphs}

\author{
{\sl Bruce Reed}\thanks{Institute of Mathematics, Academia Sinica, Taiwan; \texttt{bruce.al.reed@gmail.com}.} \and
{\sl Yelena Yuditsky}\thanks{Université libre de Bruxelles; \texttt{yuditskyL@gmail.com}.}
}

\begin{document}

\maketitle

\begin{abstract} 
We prove that for every tree $T$ which is not an edge, for almost every graph $G$ which does not contain $T$ as an induced subgraph, $V(G)$ has a partition into $\alpha(T)-1$ parts certifying this fact. Each part induces a graph which is $P_4$-free and has further properties which depend on $T$. As a consequence we obtain good bounds (often tight up to a constant factor) on the number of
$T$-free graphs and show in a follow-up paper~\cite{RY} that almost every $T$-free graph $G$ has chromatic number equal to the size of its largest clique.
\end{abstract}

\section{Overview}

In this paper we focus on the class of graphs\footnote{We assume our readers are familiar with the definition of a graph and the basic notation of graph theory. For notation which is used without being defined see \cite{BMbook}.} which do not contain some tree $T$ as an induced subgraph, such graphs are called \textit{$T$-free}. We show that the vertex set of almost every\footnote{a property holds for almost every graph in a family if the proportion of graphs of size $n$ in the family for which it fails goes to zero as $n$ goes to infinity.} $T$-free graph  has a certain kind of  partition which certifies that it is $T$-free. Results of the same flavor were previously obtained for a triangle (\cite{EKR76}), a cycle of length 4 or 5 \cite{PS91,PSBerge}, longer odd cycles and larger cliques(\cite{BB11,KPR87}), and longer even cycles \cite{R}, \cite{KKOT15} .

We start by giving some intuition behind our results. Letting $\alpha(T)$ be the size of the largest stable set in $T$, we know  $T$ cannot be partitioned into $\alpha(T)-1$ cliques. Hence any graph $G$ whose vertex set can be partitioned into $\alpha(T)-1$ cliques is $T$-free.  For every $T$ such that $\alpha(T)>\frac{|V(T)|+1}{2}$, almost every $T$-free graph 
has such a partition  which certifies it is $T$-free. We will give a very short and very straightforward proof that this 
result follows from the main result in \cite{BB11}. 

If $T$ has a near perfect matching then  almost every $T$-free graph  has such a partition unless $T$ is a subdivided
star. If $T$ is a subdivided star, then  almost every $T_4$ graph has either such a partition or a partition into 
a stable set and $\alpha-2$ cliques certifying it is $T$-free. Again we give a very short and very  straightforward proof 
that these result follows from the main result in \cite{BB11}.

If $T$ has a perfect  matching (and hence $\alpha(T) = \frac{|V(T)|}{2}$) then it  is not  true that almost every $T$-free graph has such partition. 
Consider, for example  the tree $M_6$ of Figure \ref{fig:M}. We note, $\alpha(M_6)=3$. Now there are at most 
$2^n2^{n^2/4}$ graphs on $n$ vertices whose vertex set can be partitioned into two cliques.  However,
$V(M_6)$ cannot be partitioned into two sets each of which either is an induced  subgraph of $P_3$ or contains a triangle.  
So, if  $V(G)$ can be partitioned into two sets each inducing  the complement of a  matching then it is $M_6$-free. 
For even $k$,  there are $\frac{k!}{2^{k/2}(\frac{k}{2})!}=2^{\omega(k)}$
matchings within a set of $k$  vertices. So  considering a single partition of $V(M_6)$ into $X_1,X_2$ with $|X_1-X_2| \le 1$ we see that there are $2^{\omega(n)}2^{n^2/4}$ $M_6$- free graphs for which $X_1$ and $X_2$ 
both induce complements of matchings. Hence it is not true that almost every $M_6$ free graph can be partitioned into two cliques.

It turns out however that the vertex set of  almost every $M_6$-free graph has a partition into two sets
each inducing a complement of a matching which certifies that it is $M_6$-free. 

\begin{figure}  
    \centering
                \includegraphics[width=0.18\textwidth]{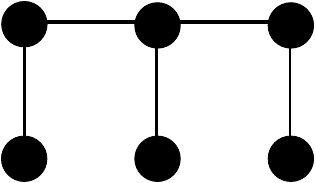}                   
                \caption{The graph $M_6$.\label{fig:M}}
\end{figure}

We prove that for every tree $T$ which is not an edge\footnote{An edge is the unique tree with $\omega(T)>\alpha(T)$
which explains its unique behaviour. If $T$ is an edge every $T$-free graph  has a partition into $\omega(T)-1=1$ sets each of which is a stable set},   almost every  $T$-free graph
has a partition into $\alpha(T)-1$ parts which certifies that it is $T$-free, 
to wit: 

\begin{definition}
For a tree $T$, a {\it $T$-freeness witnessing  partition} of a graph $G$ is a partition of $V(G)$ into $\alpha(T)-1$ parts $X_1,...,X_{\alpha(T)-1}$ such that  for any partition of $V(T)$ into  $\alpha(T)-1$ parts $Y_1,...,Y_{\alpha(T)-1}$  there is an $i\in [\alpha(T)-1]$ such that $T[Y_i]$ is not an induced subgraph of $G[X_i]$.
\end{definition}

\begin{remark}
If $G$ has a $T$-freeness witnessing partition then it is $T$-free.
Every $T$-free graph $G$  has a trivial such partition where all but 1 part is empty. 
We focus on partitions  such that each part has size at least $4^{|V(T)|}$. 
Ramsey's Theorem \cite{R30} implies that for such a partition every part contains either a clique or
or a stable set of size $|V(T)|$. 
\end{remark}

We note

\begin{claim}\label{partAll}
Let $T$ be a tree with at least four vertices, then $T$ can be partitioned into
\begin{itemize}
\item[(a)] $2$ stable sets,  
\item[(b)] $\alpha(T)$ cliques, 
\item[(c)] a stable set of size at most 2 and $\alpha(T)-1$ cliques, 
\item[(d)] an induced subgraph of $P_4$ and $\alpha(T)-2$ cliques,
\item[(e)] if $\alpha(T)>2$, an induced subgraph of $P_4$, $\alpha(T)-3$ cliques, and a stable set of size at most 2. 
\end{itemize}
\end{claim}

\begin{proof}

$T$ is bipartite so (a) holds, and  for any minimum vertex cover $C$ in $T$, there is a partition of $T$ into $\alpha(T)=|V(T)|-|C|$ cliques, $|C|$ are edges with exactly one endpoint in $C$ and the rest of which form a stable set $S$. 
Hence (b) holds  as do  (c) and (d) unless $T$ has a perfect matching. If $T$ has a perfect matching then two of the
cliques in the partition  are joined by an edge and hence span a $P_4$. We can partition this $P_4$ into a stable set and an edge 
so (c) and (d) hold.  

To prove  (e), we consider a minimal counterexample.  
We root $T$ and  take a deepest node $l$. We see that $l$ is a leaf whose parent $p$ has only leaf children.
If $p$ has two children, we delete two of them to obtain a tree $T'$ with $\alpha(T') \le \alpha(T)-1$.
Applying (d) to $T'$ yields the desired result. Otherwise, $\alpha(T-p-l)=\alpha(T)-1$.
If $\alpha(T)-1>2$, adding $pl$ to the partition showing  (e) holds for $T'$ shows (e) holds for $T$.
If $T$ is a $P_5$ or $P_6$ then deleting the endpoints of the path shows (e) holds. 
Otherwise, deleting two leaves other than $l$ from  $T$ yields a $P_3$ or $P_4$, so (e) holds. 
\end{proof}

Applying our claim  we see that we can reorder  the partition elements of any $T$-freeness witnessing  partition in which each part has size at least $4^{|V(T)|}$ as  $(X_1,\ldots X_{\alpha(T)-1}$) so that the following further  properties hold.

\begin{itemize}
\item[(i)] for each $i\ge 2, \alpha(G[X_1]) \ge \alpha(G[X_i])$, 
\item[(ii)] for each $i\ge 2$, $G[X_i]$ contains a clique of size $|V(T)|$ and $G[X_1]$ contains either a clique or a stable set  of size $|V(T)|$, and 
\item[(iii)] for each $i\in [\alpha(T)-1]$, $G[X_i]$ is $P_4$-free.
\end{itemize} 

We call such partitions {\it $T$-freeness certifying}. 

The main result of the paper is the following theorem. 

\begin{theorem}\label{main}
For any tree $T$ with $\alpha(T)>2$,  almost all $T$-free graphs $G$ have a $T$-freeness certifying partition. 
\end{theorem}

As remarked above,  a  partition into $\alpha(T)-1$ cliques is a  $T$-freeness certifying partition. By considering 
one such partition in which the size of the cliques differs by at most one we see that 
the number of $T$-free graphs on vertex set $V_n=\{1,...,n\}$ is 
$\Omega\left(2^{\left(1 -\frac{1}{\alpha(T)-1}\right)\frac{n^2}{2}}\right)$. On the other hand it is  well-known that  the 
number of $P_4$-free graphs on $l$ vertices is less than $(2l)^{2l}$ \cite{S74}. So as a consequence of 
Theorem \ref{main} , we obtain that the number of $T$-free graphs on $V_n$ is at most 
$\alpha(T)^n(2n)^{2n}2^{\left(1 -\frac{1}{\alpha(T)-1}\right)\frac{n^2}{2}}$. 

The latter upper bound strengthens a result of Pr{\"o}mel and Steger \cite{PS92} who proved that the number of 
such graphs was at most $2^{\left(1+o(1) -\frac{1}{\alpha(T)-1}\right){n \choose 2}}$. This was a special case of 
a more general theorem about the number of graphs without $H$ as an induced subgraph for 
arbitrary graphs $H$, which are called \textit{$H$-free}. 

In Section \ref{taxonomy} we present a  taxonomy of trees which allows us to  strengthen Theorem \ref{main}
by characterizing the possible $T$-freeness certifying partitions for each $T$.  
This allows us to obtain better  bounds on the number of $T$-free graphs for every $T$, which are also presented  in that section. 
We  shall also use these  results to prove in a companion paper that for every $T$,
almost every $T$-free graph $G$ satisfies $\chi(G)=\omega(G)$.  

Structural results similar to ours have been obtained when $H$ is a cycle of odd length \cite{PS92,BB11}, or a cycle of even length \cite{PS91,R,KKOT15}.  
Other related results can be found in e.g. \cite{BaBS11,BB11,EKR76,PRY18,PS91}.

One related result, due to Reed~\cite{R},  which we will state precisely below,  implies easily that 
for any tree $T$, there is an $\rho>0$ such that in almost every $T$-free graph  $G$ there is a set $Z$ of at most $|V(G)|^{1-\rho}$ 
vertices such that $G-Z$ has a $T$-freeness certifying partition. 
In Section \ref{structure} we present the Reed-Scott result and  prove a strengthening of it. 

Finally in Section \ref{betterapprox},  we prove the strengthenings of Theorem \ref{main} set out in Section \ref{taxonomy}.

We remark that if $\alpha(T)=1$ then $T$ is a vertex or an edge. In the first case  there are no $T$-free graphs,
in the second every $T$-free graph is a stable set. Furthermore,  if $\alpha(T)=2$   then  $T \in \{P_3,P_4\})$. A graph is $P_3$ free if and only if it is a disjoint union of cliques. 
A characterization of $P_4$-free graphs obtained by Seinsche is discussed below. Thus, the structure of $T$-free graphs when $\alpha(T) \le 2$ is well understood. 
For the rest of the paper,
we let $T$ be a tree with $\alpha(T) \ge 3$.

In  slightly nonstandard notation, {\it complete multipartite graph} includes stable sets which we think of as  complete multipatite graphs with one part. 

\section{Typical  Structure and Growth Rate of the $T$-free Graphs}

In this section we strengthen Theorem \ref{main} by specifying more precisely the $T$-freeness certifying partitions which 
certify that a typical $T$-free graph is $T$-free. In order to do so, we split  trees into a number of classes, the class of a tree 
determines the precise nature of the certifying partition. We also determine the growth rates of the class of $T$-free graphs for each $T$. 
\label{taxonomy}

\subsection{A Taxonomy of Trees}

In this section we present different classes of trees and set out some useful partitions of  the trees in each class.

\subsubsection{Partitioning Trees without Perfect Matchings}

We say that a graph has a {\it near perfect matching} if a matching of maximal size in the graph misses exactly one vertex. 

\begin{claim}\label{oldC} \label{oldB}
If $T$ does not have a perfect or near perfect matching then it can be partitioned into $S_2$ and $\alpha(T)-2$ cliques. 
If $T$ does not have a perfect matching then it can be partitioned into $\alpha(T)-2$ cliques and a  $P_3$.
\end{claim}

\begin{proof}
As noted in the proof of Claim \ref{partAll}, every  $T$ has a partition into $\alpha(T)$ cliques the singleton elements of which form a stable set $S$. For  $T$ without a perfect or near perfect matching, the set $S$ is of size at least $2$ which implies the  first statement in the claim. For graphs with a near perfect matching $S$ is a single vertex which has an edge to some other clique with which it forms a $P_3$, this proves the second statement. 
\end{proof}

$T$ is a {\it star} if all but one of its vertices is a leaf. $T$  is a {\it subdivided star} if it can be obtained by subdividing each edge of a star exactly once. If $T$ is a subdivided star, then there is a  vertex $v\in V(T)$ such that $T-v$ is an induced matching ($v$ is unique unless $T$ is the subdivision of an edge). Moreover $v$ is adjacent to exactly one end of each of the matching edges. We refer to $v$ as the \textit{center} of $T$.

\begin{claim}\label{oldD}
If $T$ has a near-perfect matching then it can be partitioned into  $\alpha(T)-2$ cliques and either of $P_3$ and $\overline{P_3}$. 
Furthermore, if it is not a subdivided star then it can be partitioned into $S_3$ and $\alpha(T)-2$ cliques. 
\end{claim}

\begin{proof}
If $T$ has a near perfect matching, then it has a near perfect matching $M$ where the unused vertex is a leaf $l$ (for any other perfect matching $M'$, letting $P$ be the 
longest path starting at the unused vertex whose edges are alternately in and out of the matching, we see that one  endpoint of $P$  is a leaf $l$ and swapping the edges of $M' \cap P$ with the edges of $P-M'$ yields the desired $M$). 
We let $p$ be the unique neighbour of $l$ and $pq$ the edge of $M$ containing $p$.
Then $lpq$ induces a $P_3$ and for any edge $xy$  of $M-pq$, $\{x,y,l\}$ induces a $\overline{P_3}$.  
Furthermore, if two edges of $M$ induce a $P_4$ which does not have $q$ as a midpoint (which happens if $T$ is not a subdivided star) then two edges and $l$ can be partitioned into an $S_3$ and an edge. Since $T$ is connected, the claim follows. 
\end{proof}

\subsubsection{Partitioning Trees with Perfect Matchings}  

We note that if $T$ has a perfect matching,
any such matching can be obtained by repeatedly matching a leaf with is parent and deleting them. 
So there is a unique such matching which we denote $M_T$.

A {\it spiking of} a  tree  is obtained from the disjoint union of the tree and a stable set with the same number of vertices 
by adding a perfect matching each edge of which has one end in the stable set and the other in the tree. For example,  $M_6$  is a spiking of  $P_3$. We say $T$ is a {\it spiked} tree  if it is the spiking of a subtree. 
Clearly if $T$ is spiked,  the unique tree of which it is a spiking  is the tree obtained by removing all of its 
leaves. If this subtree is a star we say that $T$ is a {\it spiked star}. 

Clearly a tree is spiked precisely if  it has a perfect matching   every edge of which  contains a leaf. 

\begin{claim}\label{ClaimM1M4}
Suppose $T$ is not a path and has  a perfect matching. Then $M_T$ contains $3$ edges $e_1,e_2,e_3$ inducing a spiked $P_3$ (which is equal to $M_6$). 
\end{claim}

\begin{proof}
Choose a vertex $v$ of degree at least 3 in $T$, let $u$ be the neighbour of $v$ such that $\{v,u\}\in M_T$. Let $x,y$ be two other neighbour of $v$ which are different from $u$. 
Then $\{v,u\}$ together with the edges in $M_T$ which contain  $x$ and $y$ induce $M_6$. 
\end{proof}

\begin{observation}\label{partM6}
$M_6$ can be partitioned into the following graphs.
\begin{itemize}
\item $(P_3,\overline{P}_3)$, $(\overline{P}_3,\overline{P}_3)$, $(S_3,S_3)$, $(S_3,\overline{P}_3)$, $(S_3,P_3)$, and 
\item $(K_2,P_4)$, $(K_2,2K_2)$, $(K_2,K_2+S_2)$, $(S_2,P_4)$, $(S_2,K_2+S_2)$. 
\end{itemize}
\end{observation}

\begin{figure} 
\centering
\includegraphics[width=0.7\textwidth]{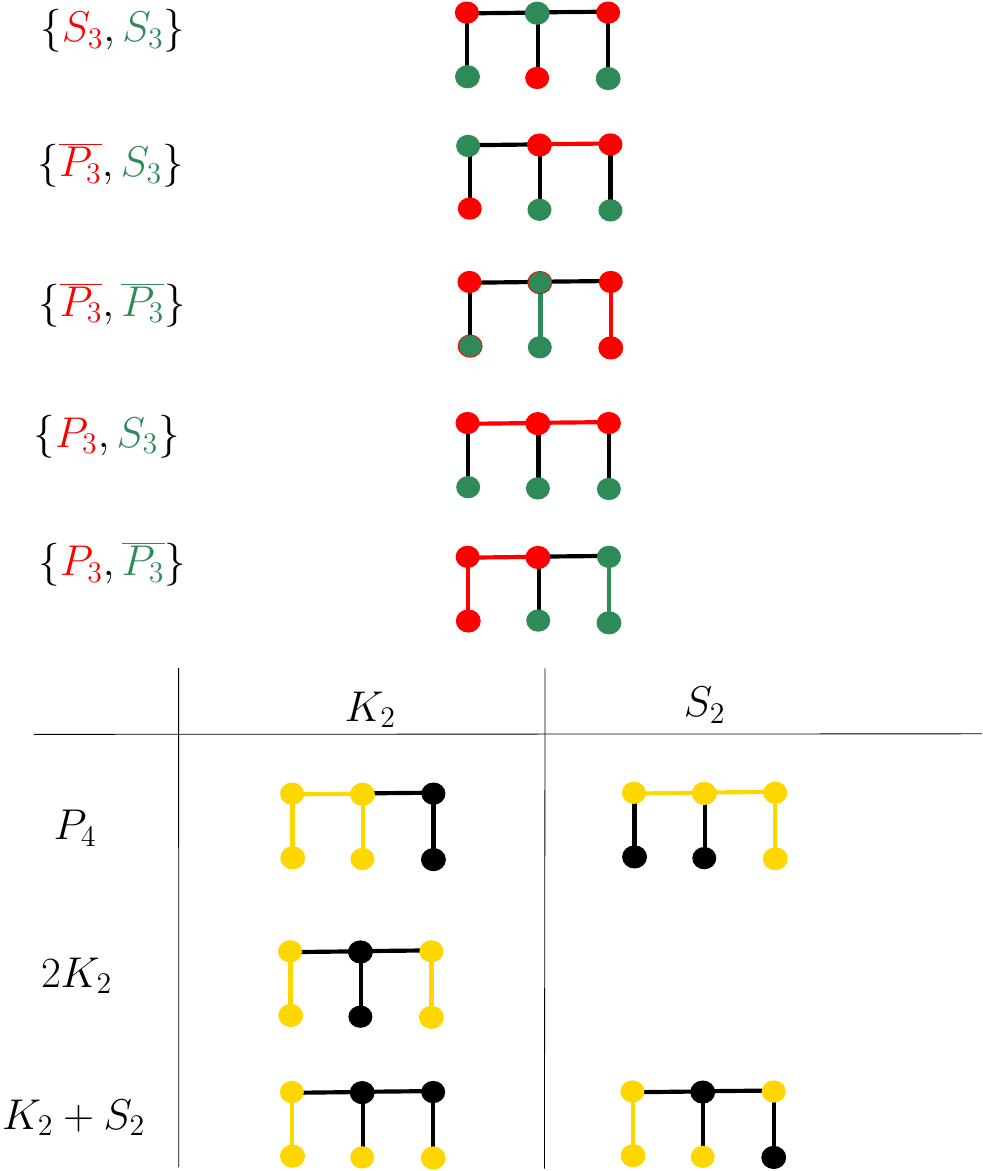}
\caption{Partitions of $M_6$.\label{P3part}}
\end{figure}

\begin{claim} \label{Claimp6contained}
Suppose $T$ is a a tree with a perfect matching which is not spiked. Then $M_T$ contains $3$ edges $e_1,e_2,e_3$ inducing a $P_6$. 
\end{claim}

\begin{proof}
Choose an edge $vu$ of $M_T$ such that neither $v$ nor $u$ is a leaf. Let $x$ be a neighbour of $v$ other than $u$ and $y$ be a neighbour of $u$ other than $x$.  Then $\{v,u\}$ together with the edges in $M_T$ which are adjacent to $x,y$ induce $P_6$.  
\end{proof}

\begin{observation}\label{partP6}
$P_6$ can be partitioned into the following graphs.
\begin{itemize}
\item $(P_3,P_3)$, $(P_3,\overline{P}_3)$, $(\overline{P}_3,\overline{P}_3)$, $(S_3,S_3)$,$(S_3,\overline{P}_3)$, and
\item  $(K_2,P_4)$, $(K_2,2K_2)$, $(K_2,P_3+S_1)$, $(S_2,P_4)$, $(S_2,2K_2)$, $(S_2,P_3+S_1)$.
\end{itemize}
\end{observation}

\begin{figure} 
\centering
\includegraphics[width=0.7\textwidth]{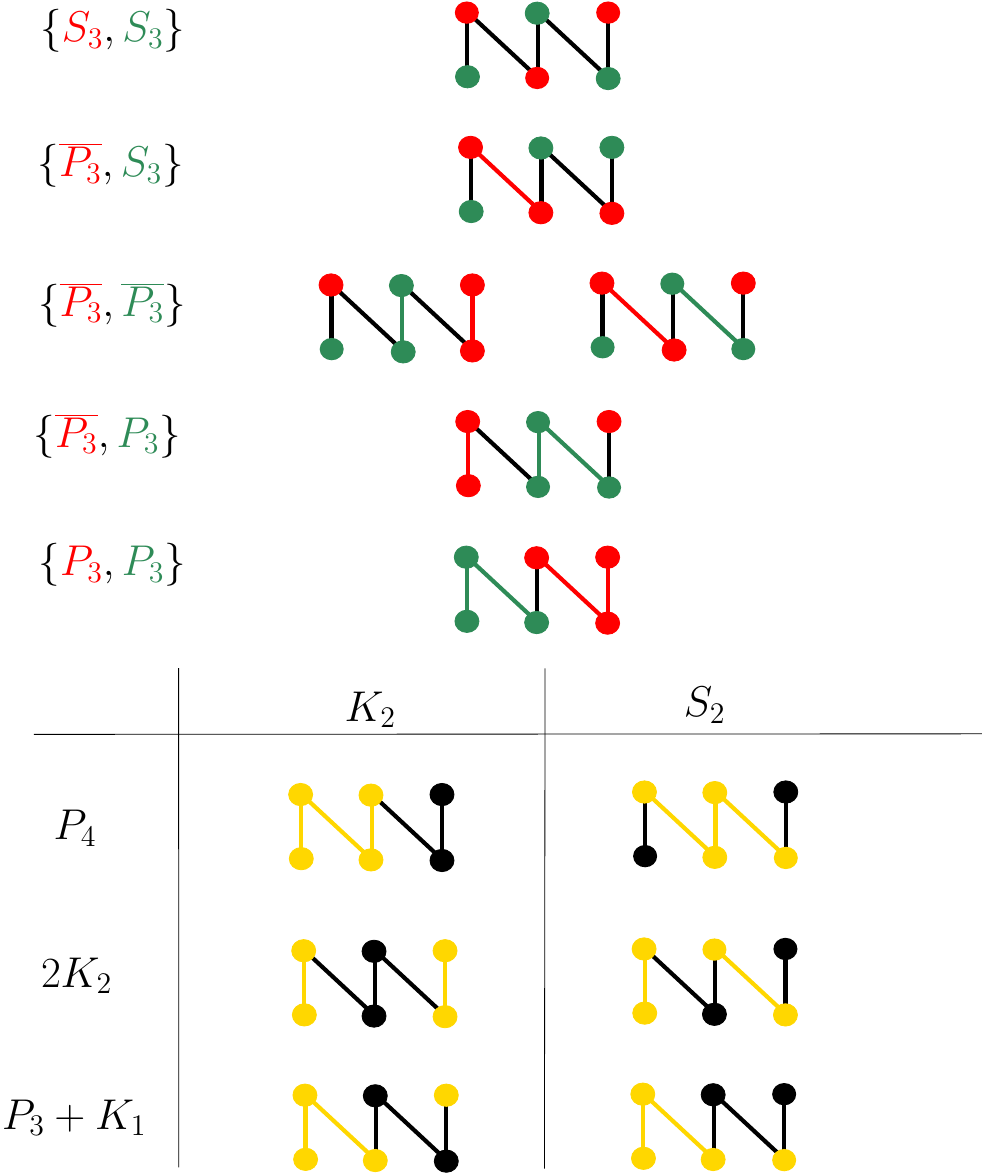}
\caption{Partitions of $P_6$.\label{P6part}}
\end{figure}

\begin{observation}\label{partcall}
Every tree $T$ with at least 6 vertices and a perfect matching can be partitioned into $\alpha(T)-3$ edges and any of the following
pairs of graphs.
\begin{itemize}
\item  $(P_3,\overline{P}_3)$, $(\overline{P}_3,\overline{P}_3)$, $(S_3,S_3)$,$(S_3,\overline{P}_3)$, and
\item  $(K_2,P_4)$, $(K_2,2K_2)$,  $(S_2,P_4)$.
\end{itemize}
\end{observation}

\begin{claim}\label{P8}
If  $T$ is an even path of length at least $8$, then for  ${\cal S}$ either of   $(P_3,S_3)$ or $(K_2,K_2+S_2)$,
$T$ can be partitioned into the graphs in ${\cal S}$  and $\alpha(T)-3$ edges.
\end{claim}

\begin{proof}
It is easy to see that $P_8$ has the above partitions and hence so does each longer even path. 
\end{proof}

A tree is a {\it chaining of two subdivided  stars} $T_1$ and $T_2$,  if it is obtained from their disjoint union by adding an edge 
between a center of $T_1$ and a center of $T_2$.  

A tree is a {\it chaining of two spiked  stars}  $T_1$ and $T_2$ with $T_1$ not an edge,  if it is obtained from their disjoint union by adding an edge 
between a  leaf adjacent to some maximum degree vertex of $T_1$ and a leaf adjacent to some maximum degree vertex of $T_2$.
We note the choice of leaf is unique unless $T_i$ is $P_2$ or $P_4$. 

A tree is a {\it Doublestar} if it is a chaining of either two spiked stars or two subdivided stars. 

\begin{claim}\label{S4}
If $T$ has a perfect matching  than it can be partitioned into an $S_4$ and $\alpha(T)-2$ edges precisely if it is neither a spiked star nor a double star.  
\end{claim} 

\begin{proof}
If a spiked star with $l>4$ leaves can be partitioned  into an $S_4$ and $\alpha(T)-2$  edges, then  one of the edges joins  a leaf and a vertex of degree two.  Deleting these vertices and this edge yields a partition of a spiked star $T'$ with $l-1$ leaves into an $S_4$ and $\alpha(T')-2$ edges. So since the spiked stars with three and four 
leaves cannot be partitioned into an $S_4$ and $\alpha(T)-2$ edges, if $T$ is a  spiked star  then it cannot be partitioned into $S_4$ and  $\alpha(T)-2$ edges. 
If $T$ is a spiking of some $T'$ which is not a star then  $T'$ contains a $P_4$. The subgraphs induced by the  edges of $M_T$ 
intersecting this $P_4$ can be partitioned into an $S_4$ and two edges, so $T$ can be partitioned into an $S_4$ and $\alpha(T)-2$
edges. 

If a doublestar   $T$ with more than  $4$  leaves can be partitioned into an $S_4$ and $\alpha(T)-2$ edges then deleting
one of these edges  which contains a leaf, yields a partition of a double star with $l-1$ leaves into an $S_4$ and $\alpha(T)-2$ edges.
It is an easy matter to verify that that no doublestar $T$ with at most $4$ leaves can be partitioned into   an $S_4$ and $\alpha(T)-2$ edges and hence this is true for every doublestar.

If $T$ has a perfect matching  and is not spiked then there are three  edges  of $M_T$ which induce a  $P_6$. If any other two edges of $M_T$  induce a 
$P_4$ then the union of the $P_6$ and the $P_4$ can be partitioned into three edges, including that joining the midpoints 
of the $P_4$, and a stable set of size 4 which includes  the endpoints of the $P_4$. In this case $T$ can be partitioned into 
an $S_4$ and $\alpha(T)-2$ edges. Otherwise, every other $e \in M_T$ contains a leaf and a vertex of degree two adjacent 
to a vertex of the $P_6$. Furthermore, if this vertex is the second or fifth vertex  of the $P_6$ then $e$ and the $P_6$ can 
be partitioned into two edges and an $S_4$, and $T$ can be partitioned into $\alpha(T)-2$ edges and an $S_4$.
Indeed a more careful analysis shows that if $T$ is not a  doublestar ( and hence the edge added between the trees is one of the edges of the $P_6$), there are two edges of $M_T$ not on the $P_6$ such that the subgraph of $T$ induced by these edges and the $P_6$ can be partitioned into an $S_4$ and three edges. Hence, $T$ can be partitioned into an $S_4$ and $\alpha(T)-2$ edges. 
\end{proof}

\begin{observation}\label{star}
 If $T$ has a perfect matching it cannot  be covered by a star and $\alpha(T)-2$ cliques. 
\end{observation}

\begin{proof}
For each node $t $ of $T$, each component $U$ of $T-t$ except one has a perfect matching. 
Thus, for any vertex $s$ of $U$, we require the same number of  cliques  to cover $U$ and $U-s$. We obtain the desired result for those stars such that every node of the star is incident to $t$. 
\end{proof}

\subsection{Characterizing $T$-freeness Certifying Partitions}

In addition to our results on partitioning trees, we will  rely on the following: 

\begin{lemma}[\cite{S74}]
Every $P_4$-free graph is either disconnected or disconnected in the complement.
Thus the components of the complement of a $P_4$-free graph $G$ induce disconnected subgraphs of $G$. 
\end{lemma}

\begin{observation}
A graph is $P_3$-free precisely if it is the disjoint union of cliques. 
A disconnected graph does not contain the disjoint union of a $P_3$ and a vertex as an induced subgraph 
 precisely if it is the disjoint union of cliques.
\end{observation}

\begin{observation}
A graph is $\overline{P_3}$-free precisely if it complete multipartite.
\end{observation}

For any graph $G$ on $[n]$,  we can reindex any partition  $\pi$ of $[n]$
so that $\alpha(G[\pi_1]) \ge \alpha(G[\pi_j])$ for all $j \ge 2$. 
We will restrict our attention to partitions where  this has been done and such that for each $i$, 
$G[\pi_i]$ contains either a clique of size $|V(H)|$ or a stable set of size $|V(H)|$.
We call such partitions {\it interesting}. 

\begin{lemma} \label{npmnss}
If  $T$ has no  perfect matching and  is not a subdivided star, then an interesting   partition $X_1,...,X_{\alpha(T)-1}$  of $V(G)$ is $T$-freeness certifying precisely if  each $G[X_i]$ is a clique. 
\end{lemma}

\begin{proof}
Clearly if each $G[X_i]$ is a clique then the partition can be reindexed so it is  $T$-freeness certifying. On the other hand, 
if the partition can be reindexed to be  $T$-freeness certifying then applying Observation \ref{oldC} if $T$ has no near perfect matching and \ref{oldB} and \ref{oldD} otherwise, we obtain that $G[X_1]$ is a clique. So, because of our reindexing, each $G[X_i]$ is a clique. 
\end{proof}

\begin{lemma}\label{ss}
If $T$ is a subdivided star then an interesting  partition $X_1...,X_{\alpha(T)-1}$  of $V(G)$ is $T$-certifying precisely  if for $i>1$, $G[X_i]$ is a clique and 
$G[X_1]$ is either a clique or a stable set.  
\end{lemma}

\begin{proof}
$T$ contains an induced matching of size $\alpha(T)-1$ which cannot be covered by a stable set and $\alpha(T)-2$ cliques. 
So  if $G[X_i]$ is a clique for $i>1$ and $G[X_1]$ is a clique or a stable set,  then the partition is $T$-freeness certifying.

By Claims \ref{oldB} and \ref{oldD} if the partition is to be $T$-freeness certifying and $G[X_1]$ is not a clique  then it is a stable set. 
If $G[X_1]$ is a clique then, because of our reindexing,  holds so is every $X_i$.  A subdivided star can be covered by its unique maximum stable set,
an $S_2$ and $\alpha(T)-3$ vertices. So, if the partition is $T$-freeness certifying and $G[X_1]$ is stable, again for all $i>1$, $G[X_i]$ is a clique. 
\end{proof}

\begin{lemma}\label{nstnds}
 If  $T$ has a perfect matching and   is not  spiked  or a doublestar, then an interesting  partition $X_1,...,X_{\alpha(T)-1}$  of $V(G)$, is $T$-freeness certifying 
precisely if  for $i>1$, $G[X_i]$ is a clique, and each  component in the complement of  $G[X_1]$ induces  in $G$, either stable sets of size three,
or the disjoint union of a clique and a vertex. 
\end{lemma}

\begin{proof}
Suppose $G[X_i]$ is a clique  for all $i>1$ and  the components in the complement of  $G[X_1]$ induce in $G$, either stable sets of size three, or the disjoint union of a clique and a vertex.   Then any induced subgraph $J$  of $G$
isomorphic to $T$ must contain four vertices of $G[X_1]$. Since $J$ is acyclic, Observation \ref{star} implies 
$J$ cannot  intersect two components  of the complement of $G[X_1]$.
Hence $J[X_1]$  must contain a triangle, contradicting its acyclicity. So our partition is $T$-freeness certifying

Suppose our partition is $T$-freeness certifying. 
If $G[X_1]$ is a clique then, because of our reindexing, each $G[X_i]$, $i>1$, is a clique.
Otherwise, if $G[X_1]$ contains a stable set of size $|V(T)|$  then because $T$ is bipartite no other $X_i$ contains a 
stable set of size $|V(T)|$ so each such $X_i$ contains a clique of size $|V(T)|$. However, this contradicts Claim \ref{S4}. So $G[X_1]$  is not a stable set and  has one of $P_3$ or $\overline{P_3}$ as a subgraph.
So,  applying Claim \ref{Claimp6contained} and Observation \ref{partP6} we see that for $i>1$, $F[X_i]$
contains no $\overline{P_3}$ or $P_3$ and since it is not a stable set must be  a clique. Since $G[X_1]$ is $P_4$-free, the components of its complement are disconnected, hence by Claim \ref{Claimp6contained} and Observation \ref{partP6} each induces in $G$, a graph whose components are cliques, at most one of which is not a vertex. Furthermore, since $P_6$ is a double star, by Claim \ref{ClaimM1M4}, Observation \ref{partM6}, Claim \ref{P8} and Claim \ref{S4} each such graph is  either a stable set of size three or the disjoint union
of a clique and a vertex. 
\end{proof}

\begin{lemma}\label{stns}
If  $T$  is spiked but  not a spiked star, then an interesting  partition $X_1,...,X_{\alpha(T)-1}$  of a graph $G$, is $T$-certifying 
precisely if either (i) $G[X_1]$ and one other $G[X_i]$ are both  the complement of a matching, and every other $G[X_i]$ is a clique,
 or (ii) for $i>1$, $G[X_i]$ is a clique, and the components in the 
complement of $G[X_1]$  induce in $G$  either stable sets of size three or the disjoint union of a vertex and the complement of a matching. 
\end{lemma}

\begin{proof} 
Suppose there are exactly two $i$ for which  $G[X_i]$  induces the complement of a matching and the other $G[X_i]$ are cliques.
Then no induced copy of $T$ in $G$ can use four vertices in any $X_i$ so it must use three vertices  inducing a $P_3$ from the two $X_i$
which do not induce cliques and an edge  from each of the other $X_i$.  But since $T$ is spiked, deleting a $P_3$ within it leaves an isolated vertex so this is 
impossible, and the partition is $T$-freeness certifying. 
Suppose  next that $G[X_i]$ is a clique for $i>1$ and the components in the complement of   $G[X_1]$ are stable sets of size three or the disjoint 
union of a vertex and the complement of a matching. Then any induced subgraph $J$  of $G$
isomorphic to $T$ must contain  at least four vertices of $X_1$. Since $J$ is acyclic, $J[X_1]$ is either the 
disjoint union of a $P_3$ and a vertex,   or  it  intersects two components 
of the complement of $G[X_1]$  and hence  must be a star. The second  of these  is  impossible by Observation \ref{star}. The 
first because deleting this disjoint union of a $P_3$ and a vertex from a spiked star leaves at least one isolated vertex.
This proves one direction of the lemma.

Suppose the partition is  $T$-certifying. 
If $G[X_1]$ is a clique then because of our reindexing, each $G[X_i]$ is a clique. 
So, we can assume $G[X_1]$ is not a clique. 
If  $G[X_1]$ contains  an $S_3$ or $\overline{P_3}$, 
applying Claim \ref{ClaimM1M4} and Observation \ref{partM6},  we see that for $i>1$, $G[X_i]$ is a clique. 
Since $G[X_1]$ is $P_4$-free, the components of its complement are disconnected, hence by Claim \ref{ClaimM1M4}, Observation \ref{partM6} and \ref{S4} each induces in $G$, a graph which is either a stable set of size three or the disjoint union
of a clique and  the complement of a matching, and we are done. Otherwise, $G[X_1]$ contains an induced $P_3$ and
 is the complement of a matching.  For $i>1$, by Claim \ref{ClaimM1M4} and Observation \ref{partM6} $G[X_i]$ is $\overline{P_3}$ free. So it is the complement of a matching. 
Now since $T$ is not a spiked star, the tree it is a spiking of contains a $P_4$ and the eight edges of $M_T$ intersecting this $P_4$ can be 
partitioned into two $P_3$s and an $S_2$. It follows that there is only one $i>1$ for which $G[X_i]$ is not a clique. 
\end{proof}

\begin{lemma} \label{spikeds}
 If $T$ is a  spiked star, an interesting  partition $X_1,...,X_{\alpha(T)-1}$  of a graph $G$, is $T$-freeness certifying 
precisely if   either  (i) for each $i$, $G[X_i]$ is the complement of a matching,  or (ii) for $i>1$, $G[X_i]$ is a clique, and the components in the 
complement of $G[X_1]$  are in $G$  the disjoint union of a vertex and a complete multipartite graph.
\end{lemma}

\begin{proof}
Since $T$ does not have two disjoint $P_3$s, we see that if each $G[X_i]$ induces the complement of a matching
then any induced copy of $T$ in $G$ would have to intersect some $G[X_i]$ in four vertices which is impossible. 
So the partition is $T$-freeness certifying. 
Suppose $G[X_i]$ is a clique for $i>1$ and the components in the complement of   $G[X_1]$ are, in $G$,  the disjoint 
union of a vertex and a complete multipartite graph.  Then any induced subgraph $J$  of $G$
isomorphic to $T$ must contain  at least four vertices of $X_1$. Since, by Observation \ref{star}, $J[X_1]$ cannot be a star it must 
be either a stable set or  the disjoint union of a vertex and  a star with at least three vertices.  
But deleting any such subgraph from $T$, along with the maximum degree vertex of the spiked  star if the subgraph deleted was a stable set, leaves 
a graph with at least $\alpha(T)-1$ components each of which is a vertex or an edge. So no such $J$ exists. This proves one direction of the lemma. 

If $G[X_1]$ is a clique then by our reindexing each $G[X_i]$ is a clique. 
Suppose the partition is $T$-freeness certifying and $G[X_i]$ contains  an $S_3$ or $\overline{P_3}$. 
Applying Claim \ref{ClaimM1M4} and Observation \ref{partM6} we see that for $i>1$, $G[X_i]$ is a clique. 
Since $G[X_1]$ is $P_4$-free, the components of its complement are disconnected in $G$. Hence, by Claim \ref{ClaimM1M4} and Observation \ref{partM6} each induces, in $G$,  the disjoint union
of a  vertex and a complete multipartite graph(which may only have one part), and we are done. Otherwise, $G[X_1]$ contains an induced $P_3$ and
 is the complement of a matching.  For $i>2$,   by Claim \ref{ClaimM1M4} and Observation \ref{partM6} 
 $G[X_i]$ is $S_3$-free, and $\overline{P_3}$-free. So it is the complement of a matching. 
\end{proof}

\begin{lemma} \label{doubles}
 If $T$ is  doublestar which is not $P_6$, an interesting  partition $X_1,...,X_{\alpha(T)-1}$  of a graph $G$, is $T$-certifying 
precisely if   for $i>1$, $G[X_i]$ is a clique, and each  component in the 
complement of $G[X_1]$  is in $G$   either stable  or the    disjoint union of a clique and a vertex. 
\end{lemma}

\begin{proof}
Suppose $G[X_i]$ is a clique for $i>1$ and the components in the complement of   $G[X_1]$ are stable sets  or the disjoint 
union of a clique and a vertex.  Then any induced subgraph $J$  of $G$
isomorphic to $T$ must contain  at least four vertices of $X_1$. If $J$ intersect two components 
of the complement of $G[X_1]$ it  must be a star. But this is impossible by Observation \ref{star}. Otherwise,
$J$ must intersect $X_1$ in  a stable set of size 4. Now,  
deleting the centres of the two stars from which $T$ arose 
 yields an induced matching with $\alpha(T)-1$ edges, which cannot be covered by a stable set and $\alpha(T)-2$
 cliques. This proves one direction of the lemma. 

If $G[X_1]$ is a clique then by our reindexing each $G[X_i]$ is a clique. 
Suppose the partition is $T$-certifying and $X_1$ is not  clique, so it  contains  an $S_3$, $P_3$, or  a  $\overline{P_3}$. 
Applying Claim \ref{ClaimM1M4}, Observation \ref{partM6}, Claim \ref{P8},   Claim \ref{Claimp6contained}, and Observation \ref{partP6} we see that for $i>1$, $G[X_i]$ is a clique. 
Since $G[X_1]$ is $P_4$-free, the components of its complement are disconnected. Hence by Claim \ref{ClaimM1M4}, Observation \ref{partM6}, Claim \ref{P8},  Claim \ref{Claimp6contained}, and Observation \ref{partP6} each component induces in $G$  either a stable set or the disjoint union
of a clique and a vertex, and we are done.
\end{proof}

Finally as is easily verified: 

\begin{lemma}\label{p6lem}
 For $T=P_6$, a relevant partition $X_1,X_2$  of a graph $G$, is $T$-certifying 
precisely if  either (i)  $G[X_2]$ is a clique and each component of  the complement of $G[X_1]$ is, in $G$,  the disjoint union of a clique and 
a stable set, or (ii) $G[X_1]$ is a stable set and $G[X_2]$ is  the complement of a matching. 
\end{lemma}

\subsection{Growth Rate of the $T$-free graphs}

In this section we determine the growth rates of the $T$-free graphs for each tree $T$ 
and the specific partition that a typical $T$-free graph has.  We set out lower bounds  on the growth rate we can prove, and upper bounds which we  will know hold once we have proven Theorem \ref{main}. 

We have already noted that we can prove that the number of $T$-free graphs is at least $2^{(1-\frac{1}{\alpha(T)-1}){n \choose 2}}$
just by considering one partition into cliques and that if Theorem \ref{main} holds then it is at most 
$2^{(1-\frac{1}{\alpha(T)-1}){n \choose 2}}2^{O(n \log n)}$. 

We want to   improve these bounds  by summing over all  $T$-freeness certifying partitions,  the number of $T$-free graphs certified by the given partition. We can determine if a partition is certifying by examining the edges within the parts. 
If it is, every choice of edges between the parts yields a $T$-free graph. 
So, for  any partition $\pi= \pi_1,...,\pi_{\alpha(T)-1}$ of $V_n$, letting $m_\pi$ be the number of pairs of vertices which lie in different partition elements and $C(\pi)$ be the number of choices of edges within the parts which makes $\pi$ a certifying partition,
the number of $T$-free graphs for which $\pi$ is a $T$-freeness certifying partition is $C(\pi)2^{m_\pi}$. 

Now,  $m_\pi={n \choose 2}-\sum_{i=1}^{\alpha(T)-1} {|\pi_i| \choose 2}$. Thus, letting $d_i=|\pi_i|-\frac{n}{w}$,
$m_\pi=(1- \frac{1}{w})\frac{n^2}{2} -\frac{\sum_{i=1}^{\alpha(T)-1} d^2_i}{2}$. Given Theorem \ref{main}, 
there are only $2^{O(n \log n)}$ choices for the edges within the parts. Thus, if the theorem is true 
almost all $T$ will be certified by partitions where $\frac{\sum_{i=1}^{\alpha(T)-1} d^2_i}{2} =O (n \log n)$
and hence for all $i$, $|\pi_i-\frac{n}{\alpha(T)-1}| =O(n^{2/3})$.

Some $T$-free graphs may have more than one partition, leading to double counting.
We can apply the following result  due to Reed(\cite{R} Claim 34) to show that the effect of such double counting is negligible.

\begin{lemma} 
\label{nodoublecount}
Let $p\ge1$ be fixed, and suppose that ${\mathcal F}_1,...,{\mathcal F}_p$ 
are families containing subgraphs of every size and  for sufficiently large $l_0$,  we have 
\begin{enumerate}
\item  The  graphs  in  
${\mathcal F}_i$  are $P_4$-free or have girth at least five 
\item 
for $i>1$ and $l\ge l_0$, every graph in ${\mathcal F}_i$ on $l$ vertices  has minimum degree at least $\frac{31l}{32}+1$;
\end{enumerate}
Then  
for every $\mu>0$, for  all sufficiently large $n$, the number of (graph, partition) pairs consisting of  
of a graph $G$ on $V_n$  and a partition of $V_n$ into  $X_1,...,X_p$ such that $|X_I-\frac{n}{p}|<n^{1-\mu}$ and  $G[X_i]$ is in ${\mathcal F}_i$ 
 for each $i$ is $\Theta(1)$ times the number of graphs on $V_n$ which  permit such a partition. 
\end{lemma}

Now, we can choose a  uniform random partition by choosing a part for each vertex indpendently 
with each part eqally likely. \
 This approach allows us to apply standard results  on the distribution of  i.i.d. random  binomial variables to obtain that for  ever fixed $w$,  and every sequence $a_1,...,a_w$ whose elements sum to $0$,  for a uniformly random choice of partition, 
$Pr(  \forall j~ d_j=a_j)=O(n^{-\frac{w-1}{2}})$ furthermore if in addition each $a_i$ lies between $-\sqrt{\frac{n}{w}}$ and $\sqrt{\frac{n}{w}}$ then  $Pr( \forall j ~d_j=a_j)=\Theta(n^{-\frac{w-1}{2}})$. 
There are at most $(2d+1)^w$ choices for a set of $a_i$ with maximum $|a_i|=d$.
It follows that the sum of $2^{m_\pi}$ over all partitions where the maximum $|d_i|$ is $d$,
is $O((2d+1)^w2^{-\frac{d^2}{2}}\frac{w^n}{n^{\frac{w-1}{2}}}2^{(1-\frac{1}{w})\frac{n^2}{2}})$.
Hence the sum of $2^{m_\pi}$ over all partitions 
is $\Theta(\frac{w^n}{n^{\frac{w-1}{2}}}2^{(1-\frac{1}{w})\frac{n^2}{2}})$ as is the sum of $2^{m_\pi}$ over all partitions where $d \le 1$. 

Combining Theorem 1.9 of \cite{BB11} with Lemmas \ref{npmnss} and \ref{ss}, we have proved Theorem \ref{main} for trees which do not have a
perfect matching. Indeed we have: 

\begin{theorem}
\label{tnpmnsss}
If $T$ has no perfect matching and is not a subdivided star then almost every $T$-free graph can be partitioned into 
$\alpha(T)-1$ cliques.  
\end{theorem}

and

\begin{theorem}
\label{tss}
If $T$ is  a subdivided star then almost every $T$-free graph can be partitioned into either 
$\alpha(T)-1$ cliques or $\alpha(T)-2$ cliques and a stable set.   
\end{theorem}

Combining this with Lemma \ref{nodoublecount} for partitions where $d<1$,  and the fact that there are 1 or 2 ways of choosing the edges within the partition elements 
to make the partition  $T$-freeness certifying for graphs without perfect matchings, we have obtained:

\begin{corollary}
\label{cnpm}
    If $T$ has no perfect matching then the number of $T$-free graphs is 
    \[\Theta\left(\frac{(\alpha(T)-1)^n}{n^{\frac{\alpha(T)-2}{2}}}2^{(1-\frac{1}{\alpha(T)-1})\frac{n^2}{2}}\right).\]
\end{corollary}

For trees with perfect matchings, the number of $T$-free trees grows more quickly, because of the number of choices for the edges 
within the partition. 

We let 
\begin{enumerate}
\item $f^1(l)$ be the number of graphs on $l$ vertices each of whose components in the complement induces in the graph the disjoint union of a 
vertex and a clique.
\item $f^2(l)$ be the number of graphs on $l$ vertices each of whose components in the complement induces  in the graph  either a stable set of size three 
or the disjoint union of a vertex and a clique. 
\item $f^3(l)$ be the number of graphs on $l$ vertices each of whose components in the complement induces  in the graph  either a stable set  
or the disjoint union of a vertex and a clique. 
\item $f^4(l)$ be the number of graphs on $l$ vertices each of whose components in the complement induces  in the graph 
 the disjoint union of a stable set  and a clique. 
\end{enumerate}

The $n^{th}$ Bell number $\Bell(n)$ is the number of ways to partition a set of $n$  (labeled) elements up into 
subsets. We have the following bounds on it:

\begin{theorem} [Equation (10) and Theorem 2.1 in \cite{BT10}] \label{Bell}
Let $n\in \mathbb{N}$, 
\[
\left(\frac{n}{e \ln n}\right)^n <\Bell(n) <\left(\frac{0.792n}{ \ln (n+1)}\right)^n.
\]
\end{theorem}

\begin{lemma}\label{graphcount}
For $i \in \{1,2,3,4\}$, $f^i(k)$  lies 
between $\Bell(k)$ and $2^k\Bell(k)$:
\end{lemma}

\begin{proof}
For any partition of the vertices  there is a graph in each  relevant class whose components in the complement  are the elements of the 
partition and such  that the components all induce in $G$ the disjoint union of a clique and a vertex. 
This proves the lower bound. 
We can specify a graph $G$  in each  relevant class by specifying the partition given by the components of 
its  complement and then specifying for each vertex whether, in $G$,  it sees any other vertex in its partition 
element.  
\end{proof}

\begin{corollary}
\label{Tfreelowerbound}
For every tree $T$  with a perfect matching there are \[\Omega\left(\frac{(\alpha(T)-1)^n}{n^{\frac{\alpha(T)-1}{2}}}\Bell( \lceil \frac{n}{\alpha(T)-1} \rceil)2^{(1-\frac{1}{\alpha(T)-1})\frac{n^2}{2}}\right)\]
$T$-free graphs. 
\end{corollary}

\begin{proof}
We simply count the extensions of  all partitions where every $d_i$ is less than 1, every component in the complement of the subgraph induced by the 
the largest partition element is a star, and the other partition elements induce cliques. We apply Lemma \ref{nodoublecount}. 
\end{proof}

\begin{remark}
For some of these families we can get a better  lower bound on the number of graphs in the family in terms of the Bell Number.  Hence we can do the same 
for the  families of $T$-free graphs for some $T$.  But since we only know the 
Bell number to within a factor of $2^{\Omega(n)}$ we have not bothered. 
\end{remark}

\begin{lemma} \label{bellratio}

For $i \in \{1,2,3,4\}$, $f^i(k)\le f^i(k+1) \le (2k+1)f^i(k)$.

\end{lemma}

\begin{proof}
For $i \in\{1,2,3\}$ this was proven in Section 2 of \cite{R}. 
For $i=4$, we note that for every graph in the relevant class on $V=\{1,...,k+1\}$, 
there is unique graph on $V'=\{1,...,k\}$ obtained by deleting $k+1$.
Conversely given a graph in the class on $V'=\{1,...,k\}$, we can obtain any graph on 
$V=\{1,...,k+1\}$ extending it by either (a) adding $k+1$ as a universal vertex, or (ii) 
adding $k+1$ to some component in the complement.  if we add $v_{k+1}$ to a component of size 1, we must make it nonadjacent to the vertex in the component. Otherwise if we add it to a component which is not a stable set in $G$,  
we can make $v_{k+1}$  adjacent in $G$  to either none of the component or exactly the vertices of the unique  nonsingleton 
clique of $G$. Finally if we add $v_{k+1}$ to a component in the complement forming a stable set of size at least 2,
we can make $v_{k+1}$ adjacent in $G$  to either  none of the component or exactly one vertex. 
This yields between 1 and $2k+1$ possibilities. 
\end{proof}

Lemma \ref{nstnds} implies that the restriction of Theorem \ref{main}
to trees with perfect matchings which are not spiked and are not double stars, is  equivalent to the following, proven below.

\begin{theorem}
\label{tss}
If $T$ is  tree with a perfect matching which is neither spiked nor a doublestar  then almost every $T$-free graph can be partitioned into 
$\alpha(T)-2$ cliques and the join of graphs each of which is either a stable set of size three  or the disjoint union of a clique and a vertex.   
\end{theorem}

Combining this with Lemma \ref{nodoublecount} and our bounds on the Bell Number  we obtain:

\begin{corollary}
\label{cnpm}
    If $T$ has a perfect matching and is neither spiked nor a double star  then the number of $T$-free graphs is   
   $$\Omega( \frac{(\alpha(T)-1)^n}{n^{\frac{\alpha(T)-2}{2}}}(\frac{n}{e(\alpha(T)-1)\ln n})^{\frac{n}{\alpha(T)-1}})2^{(1-\frac{1}{\alpha(T)-1})\frac{n^2}{2}})= \Omega( \frac{1}{\sqrt{2e}}^{n+o(n)}(\alpha(T)-1)^{n}(\frac{n}{\ln n})^{\frac{n}{\alpha(T)-1}}2^{(1-\frac{1}{\alpha(T)-1})\frac{n^2}{2}}) $$ 
    furthermore,    letting $l=\lceil \frac{n}{\alpha(T)-1}  +n^{2/3} \rceil$ if Theorem \ref{main} holds then the number of $T$-free graphs is: 
    $$O((\alpha(T)-1)^{n}2^l(\frac{0.792l}{\ln l})^l2^{(1-\frac{1}{\alpha(T)-1})\frac{n^2}{2}})= O( \sqrt{1.6}^{n}(\alpha(T)-1)^n(\frac{n}{\ln n})^{\frac{n}{\alpha(T)-1}}2^{(1-\frac{1}{\alpha(T)-1})\frac{n^2}{2}}) $$ 
\end{corollary}

\begin{proof} 
There are $\Theta(\frac{(\alpha(T)-1)^n}{n^{\frac{\alpha(T)-2}{2}}})$ partitions where each $d_i$ is less than 1. 
For each such partition, there are at least $Bell(\lceil \frac{n}{\alpha(T)-1} \rceil)$ graphs on the parts where  in one part each component is 
a stable set of size three  or the disjoint union of a clique and a vertex while the other parts are cliques..   For each such choice   there are $\Omega(2^{(1-\frac{1}{\alpha(T)-1})n^2})$
choices of edges between the partition elements  yielding a $T$-free graph. 
We now apply Lemma  \ref{nodoublecount} to obtain the lower bound. 

There are fewer than $(\alpha(T)-1)^n$  partitions  where each part is non-empty. For each such partition, the number of choices for the edges within the partition
elements where each part induces a $P_4$-free graphs is less than $2n^{2n}$. For every partition where some $d_i$ exceeds $n^{2/3}$ 
the number of extensions of such a choice to a (possibly $T$-free graph)  is $o(2^{-n^{7/6}}2^{(1-\frac{1}{\alpha(T)-1})n^2})$. 
For every partition with all $d_i<n^{2/3}$, there are at most $(\alpha(T)-1)f^2(l)$ graphs on the parts where one  part has components in the complement all of which induce in $G$  
a stable set of size three  or the disjoint union of a clique and a vertex,  and all the other parts are cliques.   For each such choice  there are $O(2^{(1-\frac{1}{\alpha(T)-1})n^2})$
choices of edges between the partition elements  yielding a $T$-free graph. 
\end{proof}

Lemma \ref{stns} tells us that to determine the structure of typical $T$-free graphs for $T$ which are spiked trees but not spiked stars
we need to compare  the number of partitions where two elements are the complements of matching and the rest are cliques with  the number of those where 
one is the join of graphs which are either stable sets of size three or the disjoint union of a clique and a vertex and the rest are cliques.
We have just seen that there are  $2^{O(n)}(\frac{n}{\ln n})^{\frac{n}{\alpha(T)-1}}2^{(1-\frac{1}{\alpha(T)-1})\frac{n^2}{2}})$ of the latter.
On the other hand

\begin{theorem}\cite{K97}
\label{K97lem}
The number of graphs on $l$ vertices which are complements of matchings is: 
\[
\left(1+\bo(l^{-1/2})\right) \left(\frac{l}{e}\right)^{l/2}\frac{e^{\sqrt{l}}}{(4e)^{1/4}}.
\] 
\end{theorem}

\begin{corollary}\label{cpmcountall}
For sufficiently large $n$ the number of graphs on $n$ vertices which can be partitioned into $w$  graphs which  are the complement of perfect matchings 
is $$\Theta\left(\frac{w^n}{n^\frac{w-1}{2}}\frac{n}{ew}^\frac{n}{2}e^{\sqrt{nw}}2^{(1-\frac{1}{w})\frac{n^2}{2}}\right).$$
\end{corollary}

\begin{proof} 

We first compare the the number of partitions of graphs in this way with the number of graphs which permit such a partition.  
We let $d$ be the max of the $|d_i|$.

The number of such partitions where $d$ is at most $1$ is
$\Omega(\frac{w^n}{n^\frac{w-1}{2}}\frac{n}{ew}^\frac{n}{2}e^{\sqrt{nw}}2^{(1-\frac{1}{w})\frac{n^2}{2}})$. 
 Applying Lemma \ref{nodoublecount} we obtain that  there are 
$\Omega(\frac{w^n}{n^\frac{w-1}{2}}\frac{n}{ew}^\frac{n}{2}e^{\sqrt{nw}}2^{(1-\frac{1}{w})\frac{n^2}{2}})$ graphs which permit 
such a partitioning.

We note that for any partition there are at most $n^{n/2}$ choices for the pattern induced by the partition. 
Thus  there are $O(w^{n}2^{-\frac{c^2}{2}}n^{n/2} 2^{(1-\frac{1}{w})\frac{n^2}{2}})$ choices  for  partitions of graphs with $d=c$. Summing up over  $c> \frac{n^{2/3}}{w}$, we obtain that the number of  partitions of graphs where $d$ is this large is 
$o(\frac{w^n}{n^\frac{w-1}{2}}\frac{n}{ew}^\frac{n}{2}e^{\sqrt{nw}}2^{(1-\frac{1}{w})\frac{n^2}{2}})$. 
Applying Lemma \ref{nodoublecount}  to the  partitions with $d \le  \frac{n^{2/3}}{w}$ we obtain number of graphs permitting  a partitioning  
into $w$ parts each inducing a complement of a matching is of the same order of the number of such partitions with $d<\frac{n^{2/3}}{w}$. 

If $d<\frac{n^{2/3}}{w}$ then $\frac{|\pi_i|}{n/w}$ is at most $1+\frac{1}{n^{1/3}}$.
Furthermore $\sqrt{|\pi_i|} \le \sqrt{\frac{n}{w}}+n^{1/6}$. it follows that the number of patterns over  a partition where each $\pi_i$ induces the 
complement of a matching and $d<\frac{n^{2/3}}{w}$  is $O(e^{O(n^{2/3})}(\frac{n}{ew})^{n/2}e^{\sqrt{nw}})$. Repeating the argument of the last paragraph,
we obtain  that the number of  of graphs permitting  a partition  
into $w$ parts each inducing a complement of a matching is of the same order of the number of such partitions with $d<n^{2/5}$. 

If $d<n^{2/5}$   then  $\sqrt{|\pi_i|} \le \sqrt{\frac{n}{w}}+1$. In this case $(\frac{|\pi_i|}{n/w})^{|\pi_i|}=\Theta(\frac{|\pi_i|}{n/w}^{n/w})=\Theta(e^{d_i})$.
Since the $d_i$ sum to 0, it follows that the number of patterns over  a partition where each $\pi_i$ induces the 
complement of a matching  and $d< n^{2/5}$ is $\Theta((\frac{n}{ew})^{n/2}e^{\sqrt{nw}})$. 

Furthermore, for every $c<n^{2/5}$ there are $O(\frac{\omega^n}{n^\frac{w-2}{2}}(2c)^w)$ partitions where $d=c$ and for each such partition,
there are $2^{-c}2^{(1-\frac{1}{w})\frac{n^2}{2}}$ choices for the edges between the partition elements. It follows that there are 
$O (\frac{w^n}{n^\frac{w-1}{2}}\frac{n}{ew}^\frac{n}{2}e^{\sqrt{nw}}2^{(1-\frac{1}{w})\frac{n^2}{2}})$ partitionings of graphs into $w$ parts such that each part induces a complement of a matching and $c \le  n^{2/5}$. 
\end{proof}

\begin{corollary}\label{cpmcounttwo}
For sufficiently large $n$ the number of graphs on $n$ vertices which can be partitioned into $w-2$ cliques and two graphs which  are the complement of perfect matchings 
is for $b=\frac{\log n}{2(1+\frac{1}{w-2})}$
$$\Theta\left((\frac{n+\frac{bw}{2}}{ew})^{\frac{n}{w}+\frac{b}{2}}e^{2\sqrt{\frac{n+(b/w2)}{w}}}2^{(1-\frac{1}{w})\frac{n^2}{2}-\frac{b^2}{2}-\frac{b^2}{2(w-2)}}\right).$$
 This is $\Omega(\frac{w^n}{n^\frac{w-1}{2}}\frac{n}{ew}^\frac{n}{w}e^{2\sqrt{{n}{w}}}2^{(1-\frac{1}{w})\frac{n^2}{2}})$ and
$O(n^{\log n}\frac{w^n}{n^\frac{w-1}{2}}\frac{n}{ew}^\frac{n}{w}e^{2\sqrt{{n}{w}}}2^{(1-\frac{1}{w})\frac{n^2}{2}})$. 

\end{corollary}

\begin{proof} 
We can assume $w>2$ as we have already treated the $w=2$ case. 

For every integer $a>0$, we consider the number of partitions where the 
size  of the two largest parts sums to $\lceil \frac{2n}{w} \rceil+a$. 
We can assume $a<n^{2/3}$, using the same argument as in the last corollary and hence by Lemma \ref{nodoublecount}, the number of 
partitioning of graphs is of the same order as the number of graphs which permit a partition.

To obtain a lower bound, we consider a partition where the two largest parts have size within $1$ and are complements of matchings
while   all the other parts have size within $1$ of each other and are cliques. 
So, $\sum_i \frac{d^2}{2}= \frac{a^2}{2}+\frac{a^2}{2(w-2))}$. The number of choices for the edges within the matching 
is $\Theta((\frac{n+\frac{aw}{2}}{ew})^{\frac{n}{w}+\frac{a}{2}}e^{2\sqrt{\frac{n+(a/w2)}{w}}})$.

To obtain an upper  bound we note that for   a given partition the number of choices for the edges between the parts is independent of which two we make 
noncliques. The choices for the edges within the parts is maximized when the two largest parts are the noncliques. So the total number of partitioning is 
of the same order as the number of partitions where the two largest parts are the noncliques.

We can assume $d=max_i \{d_i\}<n^{2/3}$, using the same argument as in the last corollary. 
This implies that  the  difference between the choices for the two non cliques 
is at most $2^{O(n^{2/3}\log n)}$ more then when $a=0$. 

Again using the same argument as in the last corollary, we obtain that we actually only need consider $d<n^{2/5}$. 
Furthermore, if for the two largest parts  we set $h_i=\frac{n}{w}+\frac{a}{2} -|\pi_i|$ $\lceil \frac{2n}{w} \rceil+a$
then  $(\frac{|\pi_i|}{\frac{n+\frac{wa}{2}}{w}})^{|\pi_i|}=\Theta(\frac{|\pi_i|}{\frac{n+\frac{wa}{2}}{w}}^{n/w})=\Theta(e^{h_i})$.
Since the $h_i$ sum to 0, it follows that the number of patterns over  any partition where each $\pi_i$ induces the 
complement of a matching  and $d< n^{2/5}$ is $\Theta((\frac{n+\frac{aw}{2}}{ew})^{\frac{n}{w}+\frac{a}{2}}e^{2\sqrt{\frac{n}{w}+\frac{a}{2}}})$. 

We extend our definition of $h_i$ by setting  for the remaining $i$, $h_i=|\pi_i|-\frac{n}{w}+\frac{a}{w-2}$. 
We note that $\mu_\pi=(1-\frac{1}{w})\frac{n^2}{2}-\frac{a^2}{2}-\frac{a^2}{2(w-2)}-\sum_i \frac{h_i^2}{2}$. 
So mimicking an earlier argument we see that the sum of $2^{m_\pi}$ over all the partitions we are considering is 
$O(2^{(1-\frac{1}{w})\frac{n^2}{2}-\frac{a^2}{2}-\frac{a^2}{2(w-2)}})$. 
Thus the number of such partitions of graphs where the size of the the  biggest two parts  sum to $\lceil \frac{2n}{w} \rceil+a$. 
 is  
$\Theta((\frac{n+\frac{aw}{2}}{ew})^{\frac{n}{w}+\frac{a}{2}}e^{2\sqrt{\frac{n+(a/w2)}{w}}}2^{(1-\frac{1}{w})\frac{n^2}{2}-\frac{a^2}{2}-\frac{a^2}{2(w-2)}})$.

Now when $a$ increases by $1$, $(\frac{n+\frac{aw}{2}}{ew})^{\frac{n}{w}+\frac{a}{2}}e^{2\sqrt{\frac{n+(a/w2)}{w}}}2^{(1-\frac{1}{w})\frac{n^2}{2}-\frac{a^2}{2}-\frac{a^2}{2(w-2)}}$ increases by $\Theta(\sqrt{n}2^{-a(1+\frac{1}{w-2})})$. So we see that this function is maximized for  $a=\frac{\log n}{2(1+\frac{1}{w-2})}+O(1)$. Furthermore, 
we see that the sum over all $a$ is of the same order as its maximum values. Hence the desired result follows. 

\end{proof}

Combined with Lemma \ref{stns}  this implies that the restriction of Theorem \ref{main}
to spiked trees which are not spiked stars  is  equivalent to the following, proven below.

\begin{theorem}
\label{tstns}
If $T$ is   spiked tree  which is not  a spiked star  then almost every $T$-free graph can be partitioned into 
$\alpha(T)-3$ cliques and two complements of a matching. Hence. letting $b=\frac{\log n}{2(1+\frac{1}{w-2})}$  the number of such graphs is 
$$\Theta\left((\frac{n+\frac{b(\alpha(T)-1)}{2}}{e(\alpha(T)-1)})^{\frac{n}{\alpha(T)-1}+\frac{b}{2}}e^{2\sqrt{\frac{n+(b(\alpha(T)-1)/2)}{\alpha(T)-1}}}2^{(1-\frac{1}{\alpha(T)-1})\frac{n^2}{2}-\frac{b^2}{2}-\frac{b^2}{2((\alpha(T)-3)}}\right).$$
\end{theorem}

Lemma \ref{spikeds} tells us that to determine the structure of typical $T$-free graphs for $T$ which are spiked  stars
we need to compare partitions where all the  elements are the complements of matching  with those where one part 
is the join of graphs each of which is the disjoint union of a complete multipartite graph and a vertex.

\begin{lemma}\label{graphcount2}
For sufficiently large $k$, the number of graphs on $k$ vertices for which all components are either a clique or obtained from the disjoint union of cliques by adding a universal vertex is at most $\left(\frac{32k}{\log \log k}\right)^k$.
\end{lemma}

\begin{proof}
We partition the vertices up into sets $S_1$ to $S_{\lceil \log k \rceil}$ where $S_i$ 
contains the vertices which are in components of the complements whose size lie  strictly 
between $2^{i-1}-1$
and $2^i$. We let $k_i=|S_i|$ and note that the vertices of $S_i$ 
are contained in a set of at most $\frac{k_i}{2^{i-1}}$ 
components.  Thus, applying our upper bound on the Bell number, for large $k$
there are at most $(\frac{k_i}{\max\{2^{i-1},\ln k_i\}})^{k_i}$ possible partitions 
of $S_i$ into these components. 

For $i=1$, the components of the complement containing vertices of $S_i$ are singletons. For $i \ge 2$,
we can specify the graph on one of these components $K$
by specifying which vertex is universal and then specifying a partition on its remaining elements.
Thus for $i \ge 2$ the number of choices for the component on $K$  is  at most $$|K|\Bell(|K|-1)\le 2^{\log |K|} (\frac{|K|}{\log |K|})^{|K|} \le (\frac{2|K|}{\log |K|})^{|K|}
\le (\frac{2^{i+1}}{i-1})^{|K|}.$$

For large $k$,  if $|k_i| \ge \sqrt{k}$ then $\frac{k_i}{\max\{2^{i-1},\ln k_i\}}\frac{2^{i+1}}{i-1}\le  \frac{16 k_i}{\log \log k} $. Otherwise $\frac{k_i}{\max\{2^{i-1},\ln k_i\}}  \le 4k_i$.

Putting this all together we obtain that the total number of graphs of the type we are counting is at most
$$\frac{k!}{k_1!k_2!..k_{\lceil \log k \rceil}!}\prod_{i=1}^{\lceil \log k \rceil} (\frac{16k_i}{\log \log k})^{k_i} (\log \log k)^{\sqrt{k}\lceil \log k \rceil} \le (\frac{32k}{\log \log k})^k.$$
\end{proof} 

This result, Lemma \ref{cpmcountall}   and Lemma \ref{spikeds} implies that the restriction of Theorem \ref{main}
to spiked stars is   equivalent to the following, proved below.

\begin{theorem}
\label{tss}
If $T$ is  a spiked star,  then almost every $T$-free graph can be partitioned into 
$\alpha(T)-1$ complements of a matching. 
\end{theorem}

Combining this with Lemma \ref{nodoublecount} we have obtained: 

\begin{corollary}
\label{cnpm}
    If $T$ is a spiked  star  then the number of $T$-free graphs is 
    \[\Theta(\frac{2^n}{n^\frac{\alpha(T)-2}{2}}\frac{n}{e(\alpha(T)-1)}^\frac{n}{2}e^{\sqrt{n(\alpha(T)-1)}}2^{(1-\frac{1}{\alpha(T)-1})\frac{n^2}{2}}).\]
\end{corollary}

Lemma \ref{doubles} implies that the restriction of Theorem \ref{main}
to double stars which are not $P_6$  is   equivalent to the following, proven below.

\begin{theorem}
\label{tss}
If $T$ is  a double star  which is not $P_6$,  then almost every $T$-free graph can be partitioned into 
$\alpha(T)-2$ cliques and a graph which is  the join of graphs each of which is either stable  or the disjoint union 
of a clique  and a vertex.    
\end{theorem}

Combining this with Lemma \ref{nodoublecount} we have obtained:

\begin{corollary}
\label{cnpm}
    If $T$ is  a double star  which is not $P_6$,  then the number of $T$-free graphs is 
    $$\Omega((\frac{1}{\sqrt{2e}})^{n+o(n)}(\alpha(T)-1)^n(\frac{n}{\ln n})^\frac{n}{\alpha(T)-1}2^{(1-\frac{1}{\alpha(T)-1})\frac{n^2}{2}})$$ and 
    $$O(({\sqrt{1.6}})^{n}(\alpha(T)-1)^n(\frac{n}{\ln n})^\frac{n}{\alpha(T)-1}2^{(1-\frac{1}{\alpha(T)-1})\frac{n^2}{2}}).$$ 
\end{corollary}

\begin{proof}
The lower bound follows from Corollary \ref{Tfreelowerbound}. and our lower bound on the Bell number 
The upper bound comes from summing $(\alpha(T)-1)^n2^{\frac{n}{\alpha(T)-1}+d}\Bell(\frac{n}{\alpha(T)-1}+d)d^{\alpha(T)-1}2^{(1-\frac{1}{\alpha(T)-1})\frac{n^2}{2}-\frac{d^2}{2}}$
over $d$ 
and using our bounds on the Bell number. 
\end{proof}

Lemma \ref{p6lem} and the upper bound on the number of graphs which are complements of matchings from $a$ implies that the restriction of Theorem \ref{main}
to $T =P_6$  is   equivalent to the following, proven below.

\begin{theorem}
\label{tss}
If $T =P_6$,  then almost every $T$-free graph can be partitioned into 
a  clique and the join of graphs each of which is the disjoint union of a clique and a stable set.    
\end{theorem}

Combining this with Lemma \ref{nodoublecount} we will obtain:

\begin{corollary}
\label{cnpm}
    If $T$ is  $P_6$   then the number of $T$-free graphs is
    $$\Omega\left((\frac{1}{\sqrt{2e}})^{n+o(n)}(\alpha(T)-1)^n(\frac{n}{\ln n})^\frac{n}{\alpha(T)-1}2^{(1-\frac{1}{\alpha(T)-1})\frac{n^2}{2}}\right)$$ and 
    $$O\left(({\sqrt{1.6}})^{n}(\alpha(T)-1)^n(\frac{n}{\ln n})^\frac{n}{\alpha(T)-1}2^{(1-\frac{1}{\alpha(T)-1})\frac{n^2}{2}}\right).$$  
\end{corollary}

\begin{proof}
The lower bound follows from Corollary \ref{Tfreelowerbound} and our lower bound on the Bell numbers. 
The upper bound comes from summing $(\alpha(T)-1)^n2^{\frac{n}{\alpha(T)-1}+d}\Bell(\frac{n}{\alpha(T)-1}+d)d^{\alpha(T)-1}2^{(1-\frac{1}{\alpha(T)-1})\frac{n^2}{2}-\frac{d^2}{2}}$
over $d$ and using our bounds on the Bell numbers. 

\end{proof}

It remains to prove the strengthenings of Theorem \ref{main} for graphs with perfect matchings we have stated 
in this section.

\section{An Approximation to Theorem \ref{main}} \label{structure}

In the rest of the paper we work only with $T$ with  a perfect matching satisfying $\alpha (T)>2$. We will not state this again.
In this section  we prove  a weakening of Theorem \ref{main} for  such $T$.

\subsection{A First Approximation  for $H$-free Graphs for Arbitrary $H$}

Our starting point is a variant  of Theorem \ref{main} proved by Reed 
which shows that we can find the desired partition if we first delete $o(n)$ vertices. 
It  also  imposes further useful   conditions on the partition obtained.

The definition of witnessing partition extends to graphs which are not trees
and this result holds for arbitrary graphs $H$

\begin{definition}
For a graph $H$, the {\it witnessing partition number of $H$}, denoted $wpn(H)$, is the maximum $t$ such that for some $s+c=t$ there is no partition of $H$ into $s$ stable sets and $c$ cliques.
\end{definition}

We note that Claim \ref{partAll} shows that if $T$ is a tree
with $\alpha(T)>1$  then $wpn(H)=\alpha(T)-1$.

\begin{definition}
We say that a partition  $X_1,\ldots,X_{wpn(H)}$ of $V(G)$ is an {\it $H$-freeness witnessing partition} if for every  partition $Y_1,\ldots,Y_{wpn(H)}$ of $V(H)$ there is an $i\in [wpn(H)]$ such that $H[Y_i]$ is not an induced subgraph of $G[X_i]$. 
\end{definition}

We note that Claim \ref{partAll}  also shows that if $T$ is a tree  with $\alpha>1$ then the two definitions of 
$T$-freeness witnessing partitions coincide.

\begin{theorem} [ \cite{R} Theorem 24] \label{thmRS}
For every graph $H$ with $wpn(H) \ge 2$ and constant $\epsilon>0$, there are $\rho=\rho(H,\epsilon)>0$ and $b=b(H,\epsilon)\in \mathbb{N}$, such that the following holds. 

For almost all $H$-free graphs $G$ with $V(G)=[n]$, there exists a partition $(\pi_1,\pi_2,...,\pi_{\wpn(H)})$ of $V(G)$ 
and a set $B\subset V(G)$ of at most $b$ vertices for which 
  we can partition
$\pi_i$ into  $Z_i$  and $\pi'_i$  such that setting 
$Z=\cup_{i=1}^{\wpn(H)} Z_i$,
\begin{itemize}
\item [(I)] the partition $(\pi'_1,\pi'_2...,\pi'_{\wpn(H)})$ is  an $H$-freeness witnessing   partition of $G-Z$,
\item [(II)] $|Z|\le n^{1-\rho}$, 
\item [(III)] for every $i\in [\wpn(H)]$ and vertex $v\in \pi_i$, there is a vertex $v' \in B \cap \pi_i$ 
such that $$\big\vert \left(N(v)\triangle N(v')\right) \cap\pi_i\big\vert \le \epsilon n,$$
\item [(IV)] for every $i\in [\wpn(H)]$, we have $$\abs{ |\pi_i|-\frac{n}{\wpn(H)}} \le n^{1-\frac{\rho}{4}}.$$
\end{itemize} 
\end{theorem}  

We note that (IV) implies that for large $n$, each part has size exceeding $4|V(T)|$, so when $T$ is a tree with $\alpha(T) \ge 3$, we can and do reindex the parts   so that the partition $(\pi'_1,\pi'_2...,\pi'_{\alpha(T)-1})$ of $G-Z$ is 
$T$-freeness certifying. 

\subsection{Stronger Partition Properties for $T$-free graphs}

In this section, we present further properties of the partition of a typical  $H$-free graph 
guaranteed to exist by Theorem \ref{thmRS}. Among other things, this allows us to show that for $T$-free graphs 
we can find a set $Y$ of $O(\log n)$ vertices such that for each $i$, $\pi_i-Y$ 
induces a $P_4$-free graph. 

We begin with some definitions. 

For every $c'$ and $s'$ with $s'+c'=wpn(H)-1$ we let $Safe_{H,c',s'}$ be the hereditary family 
consisting of those graphs $J$ for which we cannot partition $V(H)$ into $c'$ cliques, $s'$ stable sets 
and an induced  subgraph of $J$.  We let  $Safe^n_{H,c',s'}$ be the graphs in this family on the vertex set $[n]$.
We note that  Claim \ref{partAll} (a),(d),(e) 
implies  that for every tree $T$ with  $\alpha(T)>2$, 
$Safe_{H,c',s'}$ is a subfamily  of the $P_4$-free graphs. 
We let $s^n_H=\max_{\{(c',s')~ s.t. ~c'+s'=wpn(H)-1,j \le n\}} |Safe^j_{H,c',s'}|$.

 Let $G$ be a graph on $V$, we call the collection of graphs $(G[\pi_1],G[\pi_2],\ldots,G[\pi_w])$ the {\it projection} of $G$ on $\pi$. We call an arbitrary collection of graphs $F=(F_1,F_2,\ldots,F_w)$ a {\it pattern} on $\pi$ if $V(F_i)=\pi_i$ for each $i\in[w]$. We say that a graph $G$  {\it extends} a pattern $F$ on $\pi$ if the projection of $G$ on $\pi$ is equal to $F$.

Let $\epsilon>0$, we call a pattern on a partition $\pi$ of $[n]$ {\it $\epsilon$-relevant} for a graph $H$ 
if it has $wpn(H)$ parts and is the projection of some $H$-free graph $G$  onto  $\pi$   such that there are corresponding  $Z$ and $\pi'_i$ satisfying (I)-(IV). 
By Theorem \ref{thmRS} for every graph $H$ and  $\epsilon>0$,  almost every $H$-free graph $G$ on $[n]$ 
is the extension of an 
$\epsilon$-relevant pattern of $[n]$.

We will choose $\epsilon$ sufficiently small to 
satisfy certain inequalities given below.  Actually we choose a different constant $\mu>0$ sufficiently 
small in terms of $H$ and then chose $\epsilon$ sufficiently small in terms of $\mu$. 

\label{npatterns}
\begin{lemma}
 For every $H$ and $\epsilon>0$, there is a $\rho>0$ such that the number of $\epsilon$-relevant patterns  on 
 $[n]$ is at most $(s^{\lfloor 10n/9w \rfloor}_H)^w2^{\bo(n^{2-\rho})}$ where $w=wpn(H)$. 
\end{lemma}
\begin{proof} 
Having specified a partition of $[n]$ into $w=wpn(H)$ parts, we can specify the choices for the patterns on it and a set Z  for which  (I)-(IV) are satisfied 
 by specifying the choice of the vertices of $B$, their neighbourhoods, the choices of the vertices in $Z$, 
the neighbourhoods  in $F_i$ of each vertex  $z \in Z_i-B$, and the  disjoint union of the graphs $F_i[\pi'_i]$. 
 Exploiting (III) we see that we can specify the neighbourhood of a vertex of $v$ in some $\pi_i-B$  by choosing a 
vertex $v'$ in $B \cap \pi_i$ for which (III) holds, and specifying the symmetric difference of $N(v) \cap \pi_i$ and $N(v') \cap \pi_i$.
Since (I) and  (IV) hold,   $F_i[\pi'_i]$ is in $Safe^j_{H,c',s'}$ for some $c'+s'=wpn(H)-1$ and $j\le \frac{10n}{9w}$. So,  
there are fewer than $(s^{\lfloor 10n/9w \rfloor}_H)^w$ choices for the graph which is the disjoint 
union of the $F_i[\pi'_i]$. Putting this all together we see that for some $\rho>0$, the number of $\epsilon$-relevant patterns on $[n]$ is at most 
$$ w^nn^b2^{bn}n^{|Z|}b^{|Z|}2^{|Z|\epsilon n}(s^{\lfloor 10n/9w \rfloor}_H)^w =2^{O(n \log n)+|Z|(\epsilon n+\log n+\log b)} 
(s^{\lfloor 10n/9w \rfloor }_H)^w=2^{O(n^{2-\rho})} (s^{\lfloor 10n/9w \rfloor}_H)^w.$$ 
\end{proof}

\begin{corollary}
\label{firstbound}
 For every $H$ and $\epsilon>0$, there is a $\delta>0$ such that  there are  at most $O(2^{n^{2-\delta}})$ $\epsilon$-relevant patterns
 on $[n]$. 
\end{corollary}

\begin{proof}
The definition of the witnessing partition number implies that for any $c'+s'=wpn(H)-2$, each of the bipartite graphs, 
the complements of bipartite graphs, and the split graphs\footnote{graphs whose vertex set can be partitioned into a clique and a stable set} contain a graph $J$ such that $V(H)$ can be partitioned into $c'$ cliques, $s'$ stable sets, and 
$J$. Thus the witnessing partition number of $Safe_{H,c',s'}$ is at most 1.  It follows from Corollary 
8 in \cite{ABBM11} that there is a $\delta'$ such that $s^n_H \le 2^{n^{2-\delta'}}$. The result now follows from the lemma. 
\end{proof}

We note that this corollary combined with the fact that there are at least  $2^{(1-\frac{1}{wpn(H)}){n \choose 2}-O(n)}$
$H$-free graphs on $n$ vertices implies that almost every $H$-free graph on $n$ vertices extends an $\epsilon$-relevant pattern which 
has at least  $2^{(1-\frac{1}{wpn(H)}){n \choose 2}-O(n^{2-\delta})}$ extensions which are $H$-free. 

 So we can restrict our focus to such  patterns which we call {\it really $\epsilon$-relevant}. We obtain a bound on  the number of such patterns
 which is better than our bound on the number of $\epsilon$- relevant patterns. This   allows us to restrict further the number of 
 patterns which  have enough extensions to matter. We essentially iterate this process to prove our 
 main result. 

To this end, we say that a graph $J$ is {\it  $q$-pervasive} in a graph $F$  if there is collection $\cJ$ of $q$ disjoint sets of vertices of $F$ such that for each $Y\in \cJ$, $F[Y]$ is isomorphic to $J$. We call a subset $D \subseteq V(F_i)$  {\it $q$-dangerous} with respect to $H$ for the pattern $(F_1,F_2,\ldots,F_w)$ if there is a partition $(S_1,S_2,\ldots,S_w)$ of $V(H)$ such that $H[S_i]$ is isomorphic to $F_i[D]$ and  for each $j \in [w]-i$, $H[S_j]$ is $q$-pervasive in $F_j$.

We will show that the existence of a $q$-dangerous set in a pattern significantly lowers the number of $H$-free 
extensions that it has.  We exploit the following folklore observation, which is Lemma 2.12 of \cite{BB11} where 
its simple  proof is provided.     

\begin{observation}\label{lemma:edgedisK}
  Let $j$ be an integer and $r<j$ be a prime, then there are  $r^2$ edge-disjoint cliques with  $j$ vertices 
 in the complete $j$-partite graph where each part is of size $j$.   
\end{observation}

We obtain the following two results:

\begin{lemma}
Suppose, for some $q$ and $w$, we are given a partition $\pi=(\pi_1,\pi_2,\ldots,\pi_w)$ of $[n]$, a pattern $(F_1,F_2,\ldots,F_w)$ on $\pi$ and  a partition of $H$ into $S_1,\ldots,S_w$ such that for each $i$, $H[S_i]$ is 
$q$-pervasive in $F_i$. Then,  the number of choices for the edges between
the parts which yield an extension of the pattern to an $H$-free graph is at most $2^{m_\pi-\alpha  q^2}$ for some $\alpha>0$ which depends only on $H$. 
\end{lemma}

\begin{proof}
If $G$ is to be $H$-free then $w$ must be at least two. 
We know there is a prime $r$ which lies between $\frac{q}{2}$ and $q$. 
So, applying Lemma \ref{lemma:edgedisK} we see that there is a collection $\cY$ of sequences of  at least $\frac{q^2}{4}$ sets of vertices $(Y_1,Y_2,\ldots,Y_w)$ such that (i)  $Y_i\subseteq \pi_i$, (ii)  $F_i[Y_i]$ is isomorphic to $H[S_i]$, and (iii) there are no two different sequences $(Y_1,Y_2,\ldots,Y_w)$ and $(Y'_1,Y'_2,\ldots,Y'_w)$ in $\cY$ such that there is a pair $\{v,u\}$ where $v\in Y_i\cap Y'_i$ and $u\in Y_j\cap Y'_j$ for some $i\neq j\in [w]$. For each sequence $(Y_1,Y_2,\ldots,Y_w)\in \cY$ there are at most $\prod_{1\le i<j\le w}|Y_i|\cdot|Y_j|\le 2^{|V(H)| \choose 2}$ choices for edges which extend the sequence. At least one of such choices gives rise to an induced copy of $H$ and hence cannot appear. Hence the total number of ways of  extending the pattern so that an $H$-free graph is obtained  is at most 
$2^{m_\pi}\left(1-\frac{1}{2^{|V(H) \choose 2}}\right)^{q^2/4}= 2^{m_\pi-\alpha  q^2}$ for some $\alpha>0$ which depends only on $H$, as required. 
\end{proof} 

For a  pattern  $(F_1,F_2,\ldots,F_w)$  on a partition $\pi=(\pi_1,\ldots,\pi_w)$ of $[n]$ and a subset $D$ of some $\pi_k$,
we say that a choice of edges from $D$ to $V-\pi_k$ is {\it $q$-choice destroying} if  we can choose a partition $(S_1,S_2,\ldots,S_w)$ of $V(H)$ 
such that there is a bijection $f$ from $F_k[D]$ to $H[S_k]$ and for each $ j\in [w]-\{k\}$, there is a collection $\cY_j$ of $q$ vertex disjoint subgraphs of $H_j$ such that for each such subgraph $Y\in \cY_j$, $f$ can be extended to a bijection from $(F_k+F_j)[D \cup Y]$ to $H[S_k \cup S_j]$.

\begin{lemma}\label{dangerous2cor}
Suppose, for some  graph $H$ and some $q$ and $w$,  we are given a partition $\pi=(\pi_1,\pi_2,\ldots,\pi_w)$ of $[n]$ and a pattern $(F_1,F_2,\ldots,F_w)$ on $\pi$. Moreover we are given a partition of $H$ into $S_1,S_2,\ldots, S_w$ and a set $D\subset \pi_k$ for some $k\in [w]$. 
Then, the  total number of extensions of this pattern to an $H$-free graph  such that the choices for the edges from $D$ 
to $V-\pi_i$ are $q$-choice destroying   is at most $2^{m_\pi-\gamma  q^2}$ for some $\gamma>0$ which depends only on $H$. 
\end{lemma}

\begin{proof}
We bound,  separately  for each choice of  the edges of an extension of the pattern leaving $D$ which  
are $q$-choice destroying,  the number of extensions respecting this choice which are $H$-free.   

We need only consider $w\ge 3$. We recall that for every $q$ there is a prime $r$ with $\frac{q}{2} \le r \le q$.
So, by  Observation  \ref{lemma:edgedisK} there is a collection $\cY$ of at least $\frac{q^2}{4}$ sequences of sets of subgraphs $(Y_1,Y_2,\ldots,Y_{k-1},D,Y_{k+1},\ldots, Y_w)$ such that (i) $Y_i\in \cY_i$ for all $ i \in  [w]-\{k\}$ and
(ii)  there are no two different sequences $(Y_1,Y_2,\ldots,Y_{k-1},D,Y_{k+1},\ldots, Y_w)$ and $(Y'_1,Y'_2,\ldots,Y'_{k-1},D,$ $Y'_{k+1},\ldots, Y'_w)$ in $\cY$ such that there is a pair $\{v,u\}$ where $v\in Y_i\cap Y'_i$ and $u\in Y_j\cap Y'_j$ for some $i\neq j\in [w]-\{k\}$. For each $(Y_1,Y_2,\ldots,Y_{k-1},D,Y_{k+1},\ldots, Y_w)\in \cY$ there are at most $\prod_{i<j\in [w]-\{k\}}|Y_i|\cdot |Y_j|\le 2^{\binom{|V(H)|}{2}}$ choices for edges which extend this sequence respecting the given choice 
for the edges from $D$. At least one  such choice gives rise to an induced copy of $H$ and hence cannot appear. Hence the total number of ways of  extending the pattern and our choices of edges out of $D$ so that an $H$-free graph is obtained  is at most  $2^{m_\pi-|D|(|V|-|\pi_k|)}\left(1-\frac{1}{2^{|V(H)| \choose 2}}\right)^{q^2/4}= 2^{m_\pi-|D|(|V|-|\pi_k|)-\gamma  q^2}$ for some $\gamma>0$ which depends only on $H$. Considering the choices for the edges between the vertices in $D$ and $V-\pi_k$ we get the required bound.
\end{proof}

\begin{lemma}\label{dangeroustochoicedestroying}
 Let $\pi=(\pi_1,\pi_2,\ldots,\pi_w)$ be a partition of $[n]$ and let $(F_1,F_2,\ldots,F_w)$ be a pattern on $\pi$. Assume that there is a set $D\subset \pi_i$ which is $q$-dangerous for the above pattern. Then, for some   $\upsilon_H, \psi_H>0$, the number of choices of edges between $D$ and $V-\pi_i$ which are not  $\upsilon_H q  $-choice destroying is at most $w2^{|D|(n-|\pi_i|)-\psi_Hq}$.   
\end{lemma}

\begin{proof}
Let $S_1,S_2,\ldots,S_w$ be a partition of $V(H)$ which shows that $D$ is $q$-dangerous. Let $f$ be a bijections from $F_i[D]$ to $H[S_i]$. Because $D$ is $q$-dangerous, $H[S_j]$ is $q$-pervasive in $F_j$ for each $j\neq i$. Hence there are collections $\cY_j$, $j\neq i$, of disjoint sets of vertices in $\pi_j$ such that $|\cY_j|\ge q$ and for each $Y_j\in \cY_j$, $F_j[Y_j]$ is isomorphic to $H[S_j]$. For each $j\neq i$ and every  $Y_j\in\cY_j$, there is a choice of edges between $D$ and $Y_j$ for which $(F_i+F_j)[D\cup Y_j]$ is isomorphic to $H[S_i\cup S_j]$. Because we assumed that  the choice of the edges from $D$ is not $\upsilon_H q$-choice destroying,  for some $j$, the number of sets in $\cY_j$ for which we choose the above edges is at most  $\upsilon_H q$. For each $Y_j\in \cY_j$,  the probability we make this specific choice is at least $2^{-|D||Y_j|}>2^{-|V(H)|^2}$. Furthermore, each choice is independent of all the others. Thus the probability that we do not make $\upsilon_H q$ such choices is bounded above by the probability that the sum of $q$ independent variables each of which is 1 with probability $2^{-|V(H)|^2}$ is at most $\upsilon_H q$. Choosing $\upsilon_H$ sufficiently small in terms of $H$ and $\psi_H$ sufficiently small in terms of $\upsilon_H$, the result follows by well-known results on the concentration of this variable which can be obtained by using, e.g., the Chernoff Bounds \cite{C81}. 
\end{proof}

We call a graph $J$  {\it $H$-dangerous} if for every $c'$ and $s'$ with $c'+s'=wpn(H)-1$ there is a partition of $V(H)$ into $c'$ cliques, $s'$ stable sets and a set of vertices inducing a subgraph of $J$. Repeatedly ripping out a clique or 
a stable set, we see that   every graph $F$  contains either $\frac{|V(F)|-4^{|V(H)|}}{2|V(H)|}$
disjoint  sets  of size $|V(H)|$ such that either all these sets induces stables sets or  they all induce cliques. 
Thus if $J$ is $H$-dangerous then for any   pattern $F_1,...,F_{wpn(H)}$ every  subset of an $F_i$ inducing  
$J$  is $\frac{\min_{i=1}^{wpn(H)}\{|V(F_i)|\}-4^{|V(H)|}}{2|V(H)|}$-dangerous.

\begin{corollary}\label{p4hit}
    If $J$ is $H$-dangerous and for every $n$ there are at most $2^{xn}$  $J$-free graphs for some $  \omega(1)=x=o(n)$ 
    then  for some $\beta>0$ which depends only on $H$ and every sufficiently small $\epsilon>0$,  almost every $H$-free graph extends a   really $\epsilon$-relevant pattern $(F_1,...,F_{wpn(H)})$ on $[n]$ such that a maximal set of disjoint copies of $J$ each contained in some $F_i$ has at most $\beta x$ elements. Furthermore, for some $C_H$ there are at most  $2^{C_Hxn}$ such really $\epsilon$-relevant patterns. 
\end{corollary}

\begin{proof}
    For every really $\epsilon$-relevant pattern $(F_1,...,F_{wpn(H)})$ we consider a maximal set $D_1,...,D_a$  of disjoint subsets of $[n]$ each contained in some $\pi_i$ and each of which induces a copy of $J$.
    
     We can apply Lemma \ref{dangerous2cor} to show that the total number of extensions of this pattern for which for some $D_i$  the choice of edges from $D_i$ are $\lceil \frac{n}{\log n} \rceil$-choice destroying is $2^{(1-\frac{1}{wpn(H)}){n \choose 2}-\omega(n^{2-\delta})}$. Thus, almost every $H$-free graph extends a pattern in such a way that for every $D_i$, the choice of edges from $D_i$ is not  $\lceil \frac{n}{\log n} \rceil$-choice destroying. Since  every $D_i$ is $\frac{n}{3|V(H)|}$-dangerous, Lemma \ref{dangeroustochoicedestroying} implies that for every $D_i \subseteq \pi_j$, the number of choices for edges from $D_i$ to $V-\pi_j$ in such an extension is $2^{|D|(n-|\pi_j|)-\mu n}$ for  some small $\mu$ depending only on $H$. 

     If   $a \ge l=\lceil n^{1-\frac{\delta}{4}} \rceil$, we let ${\cal D}$ be the union of $D_1,..,D_l$. 
     Applying Lemmas \ref{firstbound} and \ref{dangeroustochoicedestroying}, we obtain that the number of extensions of  such patterns to $H$-free graphs is at most  
    
     $$2^{n^{2-\delta}}n^{|\cal D|}2^{-l\mu n}2^{m_\pi}.$$

     But this is $o(2^{m_\pi-n})$. Hence almost every $H$-free graphs extends a really $\epsilon$-relevant pattern with 
     $a \le l$. For every such pattern,  we set ${\cal D}=\cup_{i=1}^a D_i$.   Then $F_i-{\cal D}$ is $J$-free for every $J$. Thus, applying our bound on the speed of growth of the $J$-free graphs and Lemma \ref{dangeroustochoicedestroying}, we obtain that the number of extensions of  such patterns  for a specific value of $a$ which are $H$-free is at most 
    
     $$2^{xn}{n \choose b}2^{bn}n^{4a} 2^{m_\pi}2^{(\epsilon -a\mu) n}.$$
    
    Since $\mu$ depends only on $H$ and we can make $\epsilon$ as small as we wish, there is a $\beta$
    depending only on $H$ such that almost  every $H$-free graph extends a really $\epsilon$-relevant pattern with $a \le \beta x$. Furthermore, the number of such patterns  is at most 
         $$\sum_{a=0}^{\lfloor \beta x \rfloor}2^{xn}{n \choose b}2^{bn}n^{|\cal D|}2^{ a\epsilon n}.$$
\end{proof}

We note that since $\alpha(T)>2$, Claim \ref{partAll} (a),(d), and $(e)$ implies that $P_4$ is $T$-dangerous. 
Furthermore there are at most $2^{3 n\log n}$ $P_4$-free graphs on $n$ vertices. Hence this corollary implies the following:  

\begin{lemma} 
There is a constant $C_T$ such that  for all  sufficiently small $\epsilon >0$ almost every $T$-free graph extends an $\epsilon$-relevant pattern $F_1,...,F_{\alpha(T)-1}$ such that for some set $Y$ of at most 
$C_T \log n$ vertices, for each $i$,  $F_i-Y$ is $P_4$-free. Furthermore,  there are $2^{O(n \log n)}$ such patterns. 
\end{lemma}

  We say a pattern $\pi$ on  $[n]$ is {\it $f(n)$-extendible} if it has $f(n)$-extensions which are $T$-free. 
  
 So far in this section, we  have repeatedly shown that, for some specific $f(n)$, for almost every $T$-free graph on $[n] $ there is a partition  $\pi$ of $[n]$ such that the projection of $G$ on $\pi$ is  an $f(n)$-extendible  pattern  which is taken from a small collection of patterns $\cP$. In finishing the  proof of  Theorem \ref{main} we continue to repeatedly  show that most of  the patterns in  some such $\cP$ are actually not $f(n)$-extendible.  Hence most $T$-free graphs must be the extension of one of a much smaller collection of patterns $\cP'\subset \cP$ obtained by removing those which are not $f(n)$-extendible. This increases the function $f(n)$ for which almost   $T$-free graphs  arise from patterns  in $\cP'$ which are $f(n)$-extendable. We continue playing this game until we have proven Theorem   \ref{main}.  

Our current set of patterns $\cP$ consists of really $\epsilon$-relevant patterns for which there is a transversal of 
$Z'$ of the $P_4s$ in the pattern of size at most $C_T \log n$. We know there are $2^{O(n\log n)}$ such patterns,
so  we can restrict our attention to the subset $\cP'$ of these patterns which    extend to  $2^{(1-\frac{1}{\alpha(T)-1}){n \choose 2}-O(n\log n)}$ $T$-free graphs. Applying Lemma  \ref{dangerous2cor} we see that almost every $T$-free graph $G$ extends  a pattern for which  there is  no $n^\frac{2}{3}$-choice destroying set in  any pattern. 

One partition is a {\it coarsening} of a second if each of its parts is the union of some of the parts of the latter.  
Essentially the same proof  as that of Lemma \ref{dangerous2cor},  yields (i).
in the following lemma, and then a proof similar to that of Corollary \ref{p4hit}, yields (ii).

\begin{corollary}\label{good}
For sufficiently small $\epsilon>0$ there is a constant $C'_T\in \mathbb{N}$ such that almost every $T$-free graph $G$ on $[n]$ extends a really $\epsilon$-relevant  pattern on some partition $\pi$ of $[n]$ such that the following holds: 
\begin{itemize}
\item [(i)]  for  any coarsening of  $\pi$ and set $D$ in some part of the coarsening,  the choices of the edges out of $D$ are not  $n^\frac{2}{3}$-choice destroying set,  and
\item [(ii)] there is a set $H$ of size  at most $C'_T \log n$ consisting of the union of   a maximal family of   $\frac{n}{100(\alpha(T)-1)^2}$-dangerous set 
of size at most 8 in the coarsening, and hence intersecting  every $\frac{n}{100(\alpha(T)-1)^2}$-dangerous set  of size at most 8 in the coarsening. Furthermore for some 
$\mu>0$ which depends only on $T$ and $\epsilon$ there are $2^{O(n)+n|H|(1-\frac{1}{\alpha(T)-1}-\mu)}$ choices for $H$ and the edges from it. 
\end{itemize}
\end{corollary}

\section{The Proof of Theorem \ref{main}}\label{betterapprox}

We restrict our attention to graphs extending a really $\epsilon$-relevant pattern $F_1,...,F_{\alpha(T)-1}$  over some partition $\pi$ such that (i) and (ii) of Corollary  \ref{good}  hold. We note that since $P_4$ is $T$-dangerous 
and $\alpha(T)\ge \frac{|V(T)}{2}$, 
this implies that for all $i$, $F_i-H$ is $P_4$-free.

Each step of our analysis allows us to restrict our  attention even further. 
We first show that almost every $T$-free graph extends a pattern for which there is no $i$   such that $F_i-H$ contains a huge stable set.
We next show that the number of  bad $T$-free graphs extending a pattern for which there are  two $i$   such that $F_i-H$ contains a huge set of triples each inducing a $P_3$ 
or a $\overline{P_3}$  is a $o(1)$ proportion of the $T$-free graphs. 
We then show that the number of  $T$-free graphs extending a pattern for which there is exactly one  $i$   such that $F_i-H$ contains a huge set of triples each inducing a $P_3$ 
or a $\overline{P_3}$ is a $o(1)$ proportion of the $T$-free graphs. Having this information in hand allows us to apply a final blow to finish off the proof by showing that the number of bad $T$-free graphs extending a pattern for which there is no  $i$   such that $F_i-H$ contains either a large stable set or a huge set of triples each inducing a $P_3$ 
or a $\overline{P_3}$ is a $o(1)$ proportion of the $T$-free graphs

\subsection{Using a Large Stable Set} 

In this section, we prove the following:

\begin{lemma}\label{nobigstable}
The  $T$-free  graphs extending a really  $\epsilon$-relevant pattern $F_1,...,F_{\alpha(T)-1}$  over some partition $\pi$ such that (i) and (ii) of Lemma \ref{good} hold  and for some $i^*$ there is a stable set of size exceeding $ \frac{6n}{10(\alpha(T)-1)}$ in  $\pi_{i^*}$  
are   a $o(1)$ proportion of the $T$-free graphs on $n$ vertices. 
\end{lemma}

\begin{proof}
We assume such a stable set $S$ exists and choose it maximal.

We recall  that  deleting $Z$ yields a $T$-freeness certifying   partition.
Hence, Lemmas  \ref{nstnds} and \ref{stns} imply that $T$ is  a spiked star or a doublestar.  
Furthermore,  Lemmas \ref{spikeds},\ref{doubles}, and \ref{p6lem}  imply  that  for all $i \neq i^*$, edges are $\frac{n}{4(\alpha(T)-1)}$ pervasive in $F_i-H$. 
Moreover, $S_3$ is  $\frac{n}{100( \alpha(T)-1)^2}$-pervasive in $\pi_{i^*}$. 
So, by our choice of $H$, for $i \neq i^*$, $F_i-H$ is a clique unless  $T=P_6$ in which case $F_{3-i^*}$ is the complement of a matching. 

Since $S$ is maximal,  the component of the complement of $F_{i^*}-H$ containing $S$ is either the disjoint union of a clique 
and $S-v$  for some vertex $v$ or the disjoint union of some vertex $v$ of $S$ and a complete multipartite graph such that $S-v$ is a component of this graph.
We delete $v$ from $S$. Now, every vertex of $S$ has the same neighbours in $\pi_{i^*}$, so we can specify $S$ and 
its neighbour set  in the pattern in at most $4^{n}$ ways. 
Since $F_i-H$ is $P_4$-free there are at most $(2|\pi_{i^*}-S|)^{2|\pi_{i^*}-S|}$ choices for the subgraph $F_{i^*}-H-S$.
So there are $o(\frac{4n}{9\alpha(T)-1}^{\frac{4n}{9(\alpha(T)-1)}})$  choices for the pattern of $\pi_{i^*}$ other than the edges from $H$. 
If $T \neq P_6$ there is one choice for  the pattern other than $F_{i^*}$, otherwise since $F_{3-i^*}-H$ is the complement of a matching by Theorem  \ref{K97lem}  there are $o(|\pi_{3-i^*}|^{\frac{|\pi_{3-i^*}|}{2}})=o(n^{\frac{101n}{200(\alpha(T)-1)}})$ choices. 
 Applying Corollary \ref{good}(ii)  we have that 
the total number of choices for the edges leaving $H$ is $2^{O(n)+|H|n(1-\frac{1}{\alpha(T)-1}-\mu)}$ for some $\mu>0$ which depends on $T$. 
So, the total number of choices for graphs extending such patterns is $o(\Bell(\frac{n}{\alpha(T)-1})2^{(1-\frac{1}{\alpha(T)-1})\frac{n^2}{2}})$. 
However there are $\Omega(\Bell(\frac{n}{\alpha(T)-1})2^{(1-\frac{1}{\alpha(T)-1})\frac{n^2}{2}})$ $T$-free graphs. 
So,  we are done. 
 \end{proof}
 
 From now one, we will only consider patterns over partitions where there is  no  $i^*$  such there is a stable set of size exceeding $ \frac{6n}{10(\alpha(T)-1)}$ in  $\pi_{i^*}$.  
 We note that this implies that edges are $\frac{n}{6(\alpha(T)-1)}$-pervasive in every $F_i-H$.

\subsection{Using Two Parts Both Very Far From A Clique} 

We say $F_i-H$ is {\it very far from a clique} if it contains a family of 
 $\frac{n}{100(\alpha(T)-1)^2}$  disjoint triples which all induce a $P_3$ or all induce a $\overline{P_3}$.
 We say a $T$-free graph is {\it bad } if it does not have a $T$-freeness certifying partition.

 In this section,  we prove the following:

\begin{lemma}\label{notwofarfromclique}
The  bad $T$-free  graphs extending an $\epsilon$-relevant pattern $F_1,...,F_{\alpha(T)-1}$  over some partition $\pi$ such that we can choose $H$ so that  (i) and (ii) of Lemma \ref{good} hold  no $F_i-H$ contains  a stable set of size exceeding $ \frac{6n}{10(\alpha(T)-1)}$,   and there are two $i$ such that $F_i-H$ is very far from a clique are   a $o(1)$ proportion of the $T$-free graphs on $n$ vertices:
\end{lemma}

 \begin{proof}
 Assume the contrary and choose distinct $i_1$ and $i_2$  such that $F_{i_1}-H$ and $F_{i_2}-H$ are very far from a clique. 
 Then by the maximality of $H$, and Claims \ref{Claimp6contained}, \ref{partP6}, and  \ref{partcall}  we have (i) 
 for $j \in \{1,2\}$, $F_{i_j}-H$ contains a set $W_j$ of $\frac{n}{(100(\alpha(T)-1)^2}$  disjoint triples which all induce $P_3$. (ii)  $T$ is spiked, (ii)  every $F_i-H$ is the complement of a matching, and (iv) unless $T$ is a spiked star for 
 $i \not \in \{i_1,i_2\}$, $F_i-H$ is a clique. 
 
 So, the restriction of $\pi$ to $G-H$ is  $T$-freeness certifying on $G-H$. 
 Applying Corollary \ref{good}(ii) we have that 
the total number of choices for the edges leaving $H$ is $2^{O(n)+|H|n(1-\frac{1}{\alpha(T)-1}-\mu)}$ for some $\mu>0$ which depends on $T$. 
It follows that for some constant $D$ depending only on $H$ and $\epsilon$, the proportion of $T$-free graphs extending a partition of the type we are considering for which
$|H|>D$ are   a $o(1)$ proportion of the $T$-free graphs on $n$ vertices: Hence we need only consider choices for which $|H| \le D$.

Now, the number   of matchings on $l$ vertices which have $a$ edges is $O(2^l(2a)^{a})$. 
So by Theorem \cite{K97}, the proportion  of matchings on $l$ vertices  with more than $\frac{l}{\sqrt{\log l}}$ isolated vertices
is $o(2^{-l (\log l)^{1/3}})$. It follows that  the proportion of $T$-free graphs extending a partition of the type we are considering for which there is an $i$ such that 
$F_{i}-H$  has more than $\frac{n}{\sqrt{\log n}}$ universal vertices and either (a) $i \in \{i_1,i_2\}$ or (b) $T$ is a spiked star are    a $o(1)$ proportion of the $T$-free graphs on $n$ vertices.  Hence we need only consider (pattern,H) pairs  for which this does not hold. 

We set $t_i(v)=\min((F_i-H) \cap N(v), F_i-H-N(v))$. We say $v$ is extreme on $\pi_i$ if $t_i(v)< \frac{n}{\log \log \log n}$. 

We consider first the possibility that for some $v$, $\min_{i=1}^{\alpha(T)-1} t_i(v)> \frac{n}{\log \log n}$. 
For $j \in \{1,2\}$  our bound on the number of isolated vertices   in $F_{i_j}$, implies   $\pi_{i_j}-H$ contains either 
 $\frac{n}{10 \log \log n}$ disjoint triples with which $v$ induces a $P_4$,   or 
$\frac{n}{10 \log \log n}$ disjoint pairs  with which $v$  induces a stable set of size 3. 
Further since every $F_i-H$ is the complement of a matching, we obtain that for every $i$,
there is a family of  $\frac{n}{10 \log \log n}$ disjoint edges for  each of which which $v$ sees one endpoint and 
a family of $\frac{n}{10 \log \log n}$ disjoint  edges for  each of which which $v$ sees no  endpoint.
 It follows that  for such a choice of edges, $v$ is
$\frac{n}{10 \log \log n}$-choice destroying  for the partition $v,F_1-H,F_2-H,...,F_{\alpha(T)-1}-H$. 
So by Lemma \ref{dangerous2cor} , the  number  of $T$-free graphs  extending the type of pattern we are considering  for which such a $v$ exists are  at most 
$2^n2^{-\Omega(\frac{n^2}{(\log \log n)^2})}2^{(1-\frac{1}{\alpha(T)-1})\frac{n^2}{2}}$. Hence they are     an $o(1)$ proportion of the $T$-free graphs on $n$ vertices: Hence we need only consider choices for which no such $v$ exists and every vertex is extreme on at least one $F_i-H_i$. 

We let $DoublyExtreme$ be the subset of $H$ which are extreme on $\pi_i$ for two $i$. 
We partition $H -DoublyExtreme$ into $Extreme_1,...,Extreme_{\alpha(T)-1}$ where the vertices of $Extreme_i$ are extreme on $\pi_i$.
We note that for $v \in  F_i \cup Extreme_i$, $t_i(v)<\frac{n}{\log \log n}$.

Now, we consider the partition $ \pi'$  of $V-DoublyExtreme$ where $\pi'_i=F_i-(H -Extreme_i) \cup Extreme_i$. 
We recursively construct $H'$ by adding in sets of vertices in some $\pi'_i -H'$ which induce either 
an $S_3$ or an $\overline{P_3}$. If $T$ is not a spiked star, we also add triples in $\pi'_i-H'$ for some $i \not \in \{i_1,i_2\}$ which induce a 
$P_3$. We note that the restriction of $\pi'$ to $G-H'-DoublyExtreme$ is a $T$-freeness certifying partition of $G-H'-DoublyExtreme$. 

Thus  if  $DoublyExtreme \cup H'$ is empty, then $\pi'$ is a $T$-free witnessing partition 
for $G$ so $G$ is not bad. So we need only count graphs for which this is not the case. 

Now  for any $S \subset \pi'_i$ added to $H'$,  Lemma \ref{dangerous2cor}  implies that the number of  choices for the $T$-free graphs extending  the pattern on the $\pi'_i$ 
and a choice of the edges from $S$ to $V-\pi'_i$  which make  $S$ $n^{2/3}$-choice destroying is $2^{m_{\pi'_i}-\Omega(n^{4/3})}$. 
So summing over all $\pi'$, and patterns we see that the family of $T$-free graphs extending a choice where this occurs are a $o(1)$  proportion of the $T$-free graphs.
Applying Lemma \ref{dangeroustochoicedestroying},  and the fact that every vertex of $S$ is extreme on $\pi_i$ we obtain that  for  some $\psi>0$ and  every $S$ added to $H'$,  the number of choices for the edges from $S$ to $V$ 
is $2^{|S|(1-\frac{1}{\alpha(T)-1} -\psi)n}$.

For any $v$ added to DoublyExtreme, the number of choices for the edges from 
$v$ in $G$ is at most $2^{(1-\frac{2}{\alpha(T)-1})n+2{n \choose n/\log \log \log n}}$. 

So, since the number of complements of matchings on $l$ vertices is monotonic in $l$, we obtain that the number of  partitions of bad graphs for which 
the process set out in this section has arrived at the partition $\pi'_i$ is a $o(1)$ proportion of the number of good graphs which permit 
$\pi'_i$ as a partition. Applying Lemma \ref{nodoublecount} we are done. 
\end{proof}

\subsection{Using One  Part  Very Far From a Clique} 

 In this section,  we prove the following:

\begin{lemma}\label{nobigstable}
The  $T$-free  graphs extending an $\epsilon$-relevant pattern $F_1,...,F_{\alpha(T)-1}$  over some partition $\pi$ such that we can choose $H$ so that  (i) and (ii) of Lemma \ref{good} hold  there is no $i$ such that $F_i-H$ contains  a stable set of size exceeding $ \frac{6n}{10(\alpha(T)-1)}$,  and there is exactly one   $i$ such that $F_i-H$ is very far from a clique are   a $o(1)$ proportion of the $T$-free graphs on $n$ vertices.
\end{lemma}

 \begin{proof}
 Assume the contrary and let  $i^*$   be  such that $F_{i^*}-H$ is  very far from a clique. 
 
 Suppose first  that  $T$ is spiked.  
 The maximality of $H$ implies that for $i \neq  {i^*}$    every $F_i-H$ is the complement of a matching
 and that every component of $\overline{F_{i^*}-H}$  is either a clique or the join of a vertex and the disjoint union of cliques. 
 Lemma \ref{graphcount2} tells us that the number of choices for $F_{i^*}-H$ is at most $\frac{32|\pi_{i^*}|}{\log \log |\pi_{i^*}|}^{|\pi_{i^*}|} =\frac{n}{\alpha(T)-1}^{\frac{n}{\alpha(T)-1}}2^{-\frac{n\log \log \log n}{\alpha(T)-1} +o(n \log \log \log n)}$.  We let $b$ 
  be the sum over all $i \neq {i^*}$ of the number of edges in the matching induced by $\overline{F_i-H}$. Because there is only one part far away from a clique,
 $b \le \frac{n}{100(\alpha(T)-1)}$.  We can specify the pairs  forming nonsingleton components in the complement and hence the $F_i-H$
 for $i \neq i^*$ in at most $n^{2b}$ ways. There are $\Omega(2^{(1-\frac{1}{\alpha(T)-1})\frac{n^2}{2}}\frac{n}{e(\alpha(T)-1)}^{\frac{n}{\alpha(T)-1}})$ $T$-free graphs 
 certified by a partition where two parts are complements of a matching  and the rest are cliques. 
 Combined with Corollary \ref{good}(ii) this implies  that the proportion of $T$-free graphs which have a partition of the type we are considering where $b<\frac{n \log \log \log n}{2(\alpha(T)-1)\log n}$ 
 is a $o(1)$ proportion of the $T$- free graphs and so we need only consider graphs for which $b \ge \frac{n \log \log \log n}{2(\alpha(T)-1)\log n}$ 
 
 We let  $c$ be the number of components of $\overline{F_{i^*}-H}$ which are not stable, and let $C=\{v_1,...,v_c\}$ 
 consist of a universal vertex from each such component.  The maximality of $H$ implies $F_{i^*}-H-C$ is a complete multipartite graph. 
 We let $S$ be the union of all the parts in this graph of size exceeding $2$ and $s=|S|$.  Now, each vertex of $\pi_{i^*}-H$ misses at most one vertex of 
 $C$, so the number of choices for the vertices of $C$ and the edges from them is at most ${n \choose c}(c+1)^n$. Furthermore, the vertices of 
 $S$ see all the other vertices of $\pi_{i^*}-H-C$ and their union induces a complete multipartite graph. Thus the choices for $S$ and the edges in $F_{i^*}-H-C$ incident
 to $S$ are at most ${n \choose s}\Bell(s)$. Since $F_{i^*}-H-C-S$ is the complement of a matching, the number of choices for its edges is 
 $O(|\pi_i|^{\frac{|\pi_i|}{2}})=O(2^n\frac{n}{\alpha(T)-1}^\frac{n}{2(\alpha(T)-1)})$. 
 Since there are $\Omega(2^{(1-\frac{1}{\alpha(T)-1})\frac{n^2}{2}}\frac{n}{e(\alpha(T)-1)}^{\frac{n}{\alpha(T)-1}})$ $T$-free graphs
 these observations, combined with Corollary \ref{good}(ii) this imply  that the proportion of $T$-free graphs which have a partition of the type we are considering where 
 $c,s <4(\log n)^3$ is a $o(1)$ proportion of the $T$- free graphs and so we need only consider graphs for which one of $s$ or $c$ 
 exceeds  $4(\log n)^3 $. 
 
 But if $\max(c,s)>4(\log n)^3$ then $F_i-H$ contains $(\log n)^3$ disjoint triples each inducing an $S_3$ or an $\overline{P_3}$. Furthermore, there is a set of $b>\frac{n \log \log \log n}{2(\alpha(T)-1)\log n}$ triples each contained in some $F_i$, $i \neq i^*$ inducing a $P_3$. So, by the maximality of $H$, the number of choices for the edges between the $\pi_i$ 
 extending any such pattern to a $T$-free graph  is at most $$2^{m_\pi}\left(1-\frac{1}{2^9}\right)^{(\frac{n \log \log \log n}{2(\alpha(T)-1)\log n})(\log n)^3}=2^{m_\pi-\omega(n \log n)}.$$ Since the patterns are $P_4$-free we are done.

 So  $T$ is not spiked. and by the maximality of $H$, for $i \neq i^*$, $F_i-H$ is a clique, and the restriction of $\pi$ to $G-H$ is a $T$-freeness 
 certifying partition (we use here the fact that since there is no large stable set, if $F_j-H$ contains an $S_3$ then it contains a
 $P_3$ or $\overline{P_3}$). Applying Corollary \ref{good} (ii) we have that the total number of choices for the edges leaving $H$ is $2^{O(n)+|H|n(1-\frac{1}{\alpha(T)-1}-\mu)}$ for some $\mu>0$ which depends on $T$. It follows that for some constant $D_H$ depending only on $H$ and $\epsilon$, the proportion of $T$-free graphs extending a partition of the type we are 
  considering for which$|H|>D$ are a $o(1)$ proportion of the $T$-free graphs on $n$ vertices: Hence we need only consider choices for which $|H| \le D$.
  
  Now,  letting $k$ be the number of vertices which are in components of $F_{i^*}-H$ which are either singletons or have size exceeding 
  $n^{1/4}$  we can specify these vertices and the edges of $F_i-H$ between them in at most ${n \choose k}(1+n^{3/4})^k$ 
  ways. Hence we can specify such  patterns in at most  ${n \choose k}(1+n^{3/4})^k\Bell(\lceil \frac{n}{\alpha(T)-1}+n^{2/3}-k\rceil)$ ways. Since there are $\Omega(\Bell(\lceil \frac{n}{\alpha(T)-1} \rceil)2^{(1-\frac{1}{\alpha(T)-1})\frac{n^2}{2}})$ $T$-free graphs it follows that the number of bad graphs extending patterns where $k>\frac{n}{\sqrt{\log n}}$ is a $o(1)$ proportion of the $T$-free graphs and 
  we need only consider $k<\frac{n}{\sqrt{\log n}}$. 
 
Once again we set $t_i(v)=\min((F_i-H) \cap N(v), F_i-H-N(v))$. We now  say $v$ is extreme on $\pi_i$ if for some $\psi>0$,  $t_i(v)< \psi n$. 

We consider first the possibility that for some $v$, $\min_{i=1}^{\alpha(T)-1} t_i(v)> \frac{n}{\log \log n}$. 
Our bound on the number of  vertices in singleton or large components   of $\overline {F_{i^*}-H}$, implies   $F_{i^*}-H$ contains 
 $\omega(n^{2/3})$ disjoint triples with  all of which  $v$ induces   $J$ which is either a $P_4$ or the disjoint union of a $P_3$ and a vertex. 
Further since for $i \neq i^*$, $F_i-H$ is a clique we obtain that for every $i$,
there is a family of  $n^{2/3}$ disjoint edges for  each of which which $v$ sees one endpoint and 
a family of $n^{2/3}$ disjoint  edges for  each of which which $v$ sees no  endpoint.
 It follows that  for such a choice of edges, $v$ is
$n^{2/3}$-choice destroying  for the partition $v,F_1-H,F_2-H,...,F_{\alpha(T)-1}-H$. 
So by Lemma \ref{dangerous2cor} , the  number  of $T$-free graphs  extending the type of pattern we are considering  for which such a $v$ exists are  at most 
$2^n2^{-\omega(n^{4/3})}2^{(1-\frac{1}{\alpha(T)-1})\frac{n^2}{2}}$. Hence they are     an $o(1)$ proportion of the $T$-free graphs on $n$ vertices: Hence we need only consider choices for which no such $v$ exists and every vertex is extreme on at least one $F_i-H_i$.

We partition $H$ into $Extreme_1,...,Extreme_{\alpha(T)-1}$ where the vertices of $Extreme_i$ are extreme on $\pi_i$
but not on $\pi_j$ for $j<i$. 
We consider the partition $ \pi'$  of $V$ where $\pi'_i=(F_i-H) \cup Extreme_i$. 
We note that, by our bound on $k$, for $v \in  (F_i -H) \cup Extreme_i$, $t_i(v)<\frac{n}{\log \log n}$.
Hence  the number of choices for the edges from $H$ within the pattern on the  $\pi'_i$ is $2^{o(n)}$. 
Like $\pi$, $\pi'$  is a $T$-freeness certifying partition when restricted to $G-H$, where for $i \neq i^*$, $\pi'_i-H$ is a clique. 

if $\pi'$ is a $T$-freeness certifying partition of $G$ then $G$ is not bad, so there are no bad graphs extending such patterns. 
So, either (a)  $G[\pi'_{i^*}]$ contains a $J$ which induces a  $P_4$ or a graph with four vertices and one edge or 
(b) for some $i \neq i^*$, $G[\pi'_i]$ induces a $P_3$ or an $\overline{P_3}$ (again because of the lack of large stable sets 
if $G[\pi'_i]$ contains a stable set of size three it contains one of these graphs.)

Now, because $F_{i^*}$ is far from a clique,  $J$ is a $\frac{n}{200(\alpha(T)-1)^2}$-dangerous set. 
The number of choices of edges between the partitions yielding a $T$-free graph where it is $n^{2/3}$-choice destroying
is $2^{m_{\pi'}-\Omega(n^{4/3})}$. The number of choices of edges between the partitions yielding a $T$-free graph where it is not  $n^{2/3}$-choice 
destroying
are $2^{m_{\pi'}-\Omega(n)}$. 
\end{proof}

 \subsection{Finding a Part which is Somewhat Far From a Clique}

 In this section,  we prove the following:

\begin{lemma}\label{nobigstable}
The  $T$-free  graphs extending an $\epsilon$-relevant pattern $F_1,...,F_{\alpha(T)-1}$  over some partition $\pi$ such that we can choose $H$ so that  (i) and (ii) of Lemma \ref{good} hold  there is no   $i^*$    such that $\pi_{i^*}$  contains  a stable set of size exceeding $ \frac{6n}{10(\alpha(T)-1)}$, there is no $F_i-H$ which is very far from a clique  
 and  there is no  $i^*$ satisfying the following are   a $o(1)$ proportion of the $T$-free graphs on $n$ vertices:
 
There is a subgraph $C$  of $F_{i^*}-H$  such that   $C$ has size at least $\frac{n}{20(\alpha(T)-1)^2}$, 
every  component of $\overline{C}$ has size  at most  $n^{1-\frac{1}{5(\alpha(T)-1)}}$  and is the disjoint union of a vertex and of a clique of size at least two.
Furthermore, the components of $\overline{C}$ lie in different components of $\overline{F_{i^*}-H}$.
 \end{lemma}

 \begin{proof}

We consider  $T$-free graphs extending such a pattern  and note that edges are  $\frac{n}{12(\alpha(T)-1)}$  pervasive in every $F_i-H$. 
Thus since (ii) of Lemma \ref{good} holds, Lemmas \ref{nstnds}-\ref{p6lem} imply  that each $F_i-H$ is $P_4$-free and each component in its complement is either (i) the disjoint union of a clique 
 and a stable set, or (ii) the disjoint union of a vertex and a complete multipartite graph. 
 
 Now,  we build $W$  starting with the empty set  iteratively as follows. 
 
 (A) In iteration $j$, if   there is an $F_i$ such that  some  component $K$  in the complement of $F_i-H-W$ contains in $G$  the disjoint union of an edge and a stable set of size 2
  we proceed as follows.  We let $v_K$ be a vertex of $K$ which sees none of 
 $K-v_K$ and choose a set $S_j$ of three vertices of $K-v_K$ inducing a  $\overline{P_3}$. We add $S_j$ to $W$  and proceed to the next iteration. 
 
 (B) Otherwise,  if  there is an $F_i$ such that  some  component $K$  in the complement of $F_i-H-W$ is not,  in $G$, the disjoint 
 union of a clique  of size at least two and a vertex  we proceed as follows.  If $|K|=2$ we let $a$ and $b$ be its vertices. Otherwise, we let $v_K$ be a vertex of $K$ which sees none of 
 $K-v_K$.  $K-v_K$ is  the complement of a multipartite graph (possibly a stable set)  as otherwise we would be in (A).  We  choose two nonadjacent vertices  $a,b$ of $K-v_K$. Since the part of 
 the complete multipartite   graph $K-v_K$ containing $a$ and $b$ has at most  $ \frac{6n}{10(\alpha(T)-1)}$ vertices, we can choose a common neighbour  $c$ of $a$ and $b$
 in $\pi_i-H-W$. We can also choose $c$
 to not be the unique vertex of the component of the complement of $F_i-H-W$   containing it which misses  (in $G$) all the other vertices of the component, unless the component only has  one vertex.  We add $S_j=\{a,b,c\}$ to $W$  and proceed to the next iteration.

 If $|W| >\frac{n}{10(\alpha(T)-1)}$ then there is some $i^*$ with $|W \cap\pi_{i^*} | >\frac{n}{10(\alpha(T)-1)^2}$  and $F_{i^*}-H$ is far from a 
 clique which yields a contradiction.

 Otherwise we note that for every $i$, the components of $\overline{F_i-H-W}$ are stars of size at least 3. 
 For each $i$, we let $C_i$ be the union of the components of $\overline{F_i-H-W}$ which  have  size  between
$3$ and $n^{1-\frac{1}{(5\alpha(T)-1)}}$. If any $C_i$ has size exceeding $\frac{n}{20(\alpha(T)-1)^2}$ we set $C=C_i$ and obtain a contradiction.

Otherwise, we can specify  all the $C_i$    in at most  $2^n$ ways.
We can specify the components of   $\overline{C_i}$ in at most $|C_i|^{|C_i|}$ ways. We can specify the edges of $G$ within each component by specifying 
its unique isolated vertex. So the total number of choices for  the $C_i$ and the subgraphs induced by all the $C_i$ is at most $4^nn^{\sum_i |C_i|} \le 4^nn^{\frac{n}{10(\alpha(T)-1)}}$.
We can specify the singleton components of all the $\overline{F_i-H-W-C_i}$ in at most $2^n$ ways. The remaining components have size exceeding $n^{1-\frac{1}{5(\alpha(T)-1)}}$.
So there are at most $n^{\frac{1}{5(\alpha(T)-1)}}$  of them  and  we can specify them in at most $n^{\frac{n}{5(\alpha(T)-1)}}$ ways.  Again, we can specify the edges of $G$ within each component by specifying its unique isolated vertex. We can specify all these vertices in at most $2^n$ ways. 

 We consider   for any $j$,  $i$ between $1$ and $\alpha(T)-1$, and choice of $F_i-H-\cup_{l<j} S_l$, the number of choices for $S_j \subseteq F_i$ and the edges from it within the partition.
 We can choose $S_j$ in $n^3$ ways. Now, if (A) occurs there are at most $n$ choices for $K$. Furthermore, If $K-S_j$ is not a stable set then $K$ must be obtained from $K-S_j$ 
 by adding the edge of $S_j$ to the nonsingleton clique of $K-S_j$, and adding the remaining vertex of $S$  to the stable set formed  by the rest of $K-S_j$. Otherwise, $K$ is the disjoint union  either of 
 the edge in $S_j$ and a stable set, or of  a triangle containing this edge and a stable set. So, in any case there are at $n^2$ choices for $K$ and  the edges from $S_j$ to  $K-S_j$. If (B) occurs there are again
  at most $n$ choices for $K$(including the possibility that $a$ and $b$ are the only vertices of $K$). Furthermore, if $|K|>2$,  we either add $a$ and $b$ to a part of  $K-S_j-v_K$  or make them a new part. So there are at most $n^2$ choices for $K$ and the edges of
  $K-c$. Now, there are $n$ choices for the component $K_c$ containing $c$. There are again only $n+1$ choices for the edges from $c$ to  $K_c-c$. 
  
 So, in total  for any $j$,  over all $i$ between $1$ and $\alpha(T)-1$, and choice of $F_i-H-\cup_{l<j} S_l$, the number of choices for $S_j$ and the edges from it within the pattern
 is  less than $n^8$. 
 
 So the total number of choices for $W$ and the edges from it is less than $n^{\frac{8|W|}{3}} \le n^{1/3(\alpha(T)-1)}$. 
 
Applying Corollary \ref{good}(ii),  we have that 
the total number of choices for the edges leaving $H$ is $2^{O(n)+|H|n(1-\frac{1}{\alpha(T)-1}-\mu)}$ for some $\mu>0$ which depends on $T$. 
So there are  at most $2^{O(n)}n^{\frac{n}{3(\alpha(T)-1)}+ \frac{n}{10(\alpha(T)-1)} + \frac{n}{5(\alpha(T)-1)}} $  choices for the pattern
and $o(\Bell( \lceil \frac{n}{\alpha(T)-1} \rceil) 2^{(1-\frac{1}{\alpha(T)-1})\frac{n^2}{2}}) $ choices for the graphs which extend them 
  \end{proof}

\subsection{Using a Part Which is Somewhat Far From a Clique}

 In this section,  we  complete the proof of Theorem \ref{main} by proving the following:

\begin{lemma}\label{nobigstable}
The  bad $T$-free  graphs extending an $\epsilon$-relevant pattern $F_1,...,F_{\alpha(T)-1}$  over some partition $\pi$ such that we can choose $H$ so that  (i) and (ii) of Lemma \ref{good} hold,  there is no   $i^*$    such that $\pi_{i^*}$  contains  a stable set of size exceeding $ \frac{6n}{10(\alpha(T)-1)}$, there is no $F_i-H$ which is very far from a clique  
 and  there is an  $i^*$ satisfying the following are   an $o(1)$ proportion of the $T$-free graphs on $n$ vertices:
 
There is a subgraph $C$  of $F_{i^*}-H$  such that   $C$ has size at least $\frac{n}{20(\alpha(T)-1)^2}$, 
every  component of $\overline{C}$ has size  at most  $n^{1-\frac{1}{5(\alpha(T)-1)}}$  and is the disjoint union of a vertex and a clique of size at least two.
Furthermore, the components of $\overline{C}$ lie in different components of $\overline{F_{i^*}-H}$.
 \end{lemma}

\begin{proof}

We define a set $H'$ iteratively as follows. Initially $H'$ is empty. 
For each component $J$ of $C-H' $ we let $v_J$ be a vertex which in $G$ sees none of $J-v_J$. 
For $v\in V-F_{i^*}-H-H'$ we define $J'(v)$ to be the set of vertices of $J$ which are adjacent to $v$ if $v_J$ is not, or are nonadjacent to $v$ otherwise. 

To obtain  $H'$ we  repeatedly add to it  if possible either 
\begin{enumerate}
\item[(a)]  a vertex which  is  in $F_i-H-H'$ for some $i \neq {i^*}$  and such that either $v$ is nonadajacent to fewer than $\frac{|C-H'|}{10}$
vertices of $C-H'$ or the sum over all the components  of  $\overline{C-H'} $ of  $|J'(v)|$
is less than $\frac{|C-H'|}{10}$, or 
\item[(b)] a triple of vertices   in $F_{i^*}-H-H'$  such that  for some $i \neq {i^*}$  there is a vertex $u$ of $F_i-H-H'$ such that $u$  has at least
$\frac{n}{10(\alpha(T)-1)}$ neighbours in $F_i-H-H'$  and all of these of have an edge to the triple.  
\end{enumerate}

We note that for some $\psi>0$ depending only on $T$,  the number of choices for the vertices of $H'$ and the edges from them is $O(2^{bn}2^{|H'|(1-\frac{1}{\alpha(T)-1}-\psi)n})$. 
So, for all $T$ there is a $C'_T$ such that  the subset of the $T$-free graphs  we are considering for which $|H'|>C'_T \log n$ is a $o(1)$ proportion of the $T$-free graphs. 
Hence we can  and do restrict our attention to $G$ for which we stop growing $H'$ before it contains $C'_T \log n$ vertices.

We note that at this point, for every $v \in V-F_{i^*}-H-H'$  there  is a set $T_v$ of at least $\frac{n^{\frac{1}{6\alpha(T)-1}}}{2000(\alpha(T)-1)^2}$ disjoint triples  in $C$ which together with $v$ induce a $P_4$, in which $v$ is an endpoint. 

Suppose $T$ is spiked. 

The maximality of $H$,  implies  none of these $P_4$s extend to an  $M$ in $F_{i^*} \cup F_i-H$.
For each $i \neq i^*$,  for each component $K$ of $F_i-H-H'$ which has size exceeding 2
we apply  this fact to a vertex  $v$ in $K$ missing the rest of $K$.

We let $P_v$  be a family of $\lfloor \frac{|K|-1}{2} \rfloor>\frac{|K|}{3}$ pairs 
of vertices in $K-v$. For each pair in $P_v$ and triple in $T_v$ there is a choice of edges between them which yields that together 
with $v$ they yield an $M$. Thus, the number of choices of edges between the union of the elements of $T_v$ and the union of the elements 
of the $P_v$ is at most $2^{6|T_v||P_v|}(1-\frac{1}{2^6})^{|T_v||P_v|}= O( 2^{6|T_v||P_v|- n^{\frac{1}{7(\alpha(T)-1)}}|K|})$. 
On the other hand,  since $K-v$ is a complete multipartite,  the number of choices for $K$ and the edges from its vertices within $F_i-H-H'$  is at most ${n \choose k}n \Bell(|K|-1)
=O(2^{(2|K|+1)\log n})$. 

We consider next the  number $b$  of pairs of vertices forming  components of $F_i-H-H'$. We know  we have $b$ vertex disjoint $P_3$ each contained in $F_i$ for some $i \neq i^*$. Now, there are $n^{\frac{1}{6(\alpha(T)-1)}}$
disjoint $\overline {P_3}$ in $C$. Hence since $H$ is maximal the number of choices for edges between the elements of  $\pi$ 
is at most $ 2^{m_\pi}(1-\frac{1}{2^9})^{bn^{\frac{1}{7(\alpha(T)-1)}}}$.
It follows that the $T$-free graphs  extending a pattern we are considering extends one for which $b \ge n^{1-\frac{1}{3(\alpha(T)-1)}}$ 
is a $o(1)$-proportion of the $T$-free graphs, 
and we can focus on $G$ for which this is not  the case.
But then, since we can choose the two vertex components of the $F_i$ for $i \neq i^*$  in at most $n^{2b}$ ways, it follows that given  the choice for the pattern 
on $F_{i^*}-H-H'$, there are $2^{O(n)}$ choices for the rest of the pattern on $G-H-H'$. But by the maximality of $H$,  the components  of $\overline{F_{i^*}-H-H'}$
are the disjoint union of a vertex and a complete multipartite graph. Hence the number of $T$-free graphs extending the patterns we are considering 
is a $o(1)$ proportion of the $T$-free graphs certified by partitions into two complements of a matching and $\alpha(T)-3$ cliques.

So $T$ is not spiked. 

By the maximality of $H$,  for every $v \in F_i-H-H'$, none of the $P_4$s  induced by $v$ and a triple in $T_v$ extend to an  $P_6$ in $F_{i^*} \cup F_i-H$.
For each $i \neq i^*$,  for each component $K$ of $F_i-H-H'$ which has size at least 2. 
We apply  this fact to a vertex  $v$ in $K$ missing the rest of $K$.

If $|K| >20$, we greedily  select a family $P_v$   of $\lceil \frac{|K|}{20} \rceil$ pairs 
preferring pairs of adjacent vertices. For each  adjacent pair in $P_v$ and triple in $T_v$ there is a choice of edges between them which yields that together 
with $v$ they yield a $P_6$. If $P$ is a pair of nonadjacent vertices in $P_v$ then there is a stable set in $K$ containing all but  less than $\frac{|K|+1}{10}$ of its 
vertices. Hence $|K|<\frac{4n}{5(\alpha(T)-1)}$. Thus there are at least $\frac{n}{10(\alpha(T)-1)}$ vertices in $F_i-H-H'-K$ and this set which forms the neighbourhood of $v$ 
in $F_i-H-H'$ is also the common neighbourhood of $v$ and the vertices in $P$ in $F_i-H-H'$. By the choice of $H'$, for each triple  in $T_v$, one of these common neighbours 
sees none of the triple.  Hence again,  there is a choice of edges between $P$ and the triple  which yields that $F_i \cap F_{i^*}-H-H'$
contains  a $P_6$. Thus, the number of choices of edges between the union of the elements of $T_v$ and the union of the elements 
of the $P_v$ is at most $2^{6|T_v||P_v|}(1-\frac{1}{2^6})^{|T_v||P_v|}= O( 2^{6|T_v||P_v|- n^{\frac{1}{7(\alpha(T)-1)}}|K|})$. 
If $|K|<20$ then again we have that there are at least $\frac{n}{10(\alpha(T)-1)}$ vertices in $F_i-H-H'-K$ and this set which forms the neighbourhood of $v$ 
in $F_i-H-H'$ is also the common neighbourhood of the vertices of $K$, and hence by the choice of $H'$, for any triple in $T_v$, one of these common neighbours 
is nonadjacent to all of the triple. 
Thus, again the choices for the edges between $K$ and $F_{i^*}-H-H'$ is $2^{|K|(|\pi_{i^*}|-H-H'|-\Omega(n^\frac{1}{7(\alpha(T)-1)})}$. 
On the other hand,  since $K$ is the disjoint union of a clique and a stable set, the number of choices for $K$ and the edges from its vertices within $F_i-H-H'$  is at most ${n \choose k}n 2^{|K|}=O(2^{(2|K|+1)\log n})$. 

It  follows that given  the choice for the pattern 
on $F_{i^*}-H-H'$, there are $2^{O(n)}$ choices for the rest of the pattern. By the maximality of $H$,  the components  of $\overline{F_{i^*}-H}$
are the disjoint union of a clique and a stable set. Mimicking an earlier argument we obtain that the proportion of $T$-free graphs extending patterns we are considering where 
the number $k$ of  vertices in components of $F_{i^*}-H$ which are singletons or have size exceeding $(\log n)^2$  exceeds  $\frac{n}{\sqrt{\log n}}$ is a $o(1)$ proportion 
of the $T$-free graphs.  Hence we can and do focus on patterns for which this is not the case. 

We now forget about $H'$ and construct an $H''$ by  repeating the above argument replacing $C$ by $F_{i^*}-H$.

For each component $J$ of $F_{i^*}-H-H''$ we let $v_j$ be a vertex which in $G$ sees none of $J-v_J$. 
For $v\in V-F_{i^*}-H-H''$ we define $J'(v)$ to be the set of vertices of $J$ which are adjacent to $v$ if $v_J$ is not, or are nonadjacent to $v$ otherwise. 

We define a set $H''$ by repeatedly adding to it  if possible either 
\begin{enumerate}
\item[(a)]  a vertex which  is  in $F_i-H-H''$ for some $i \neq {i^*}$  and such that either $v$ is nonadajacent to fewer than $\frac{n}{10(\alpha(T)-1)}$
vertices of $\pi_{i^*}-H-H''$ or the sum over all the components  of  $\overline{F_{i^*}-H-H''} $ of  $|J'(v)|$
is less than $\frac{n}{10(\alpha(T)-1)}$, or 
\item[(b)] a triple of vertices   in $F_{i^*}-H-H''$  such that  for some $i \neq {i^*}$  there is a vertex $u$ of $F_i-H-H''$ such that $u$  has at least
$\frac{n}{10(\alpha(T)-1)}$ neighbours in $F_i-H-H''$ and each such neighbour has an edge to the triple.  
\end{enumerate}

We note that for some $\psi>0$. depending only on $T$, the number of choices for the vertices of $H''$ and the edges from them is $O(2^{bn}2^{|H''|(1-\frac{1}{\alpha(T)-1}-\psi)n})$.
So, for all $T$ there is a $C'_T$ such that  the subset of the   $T$-free graphs  we are considering for which $|H''|>C'_T \log n$ is a $o(1)$ proportion of the $T$-free graphs. 
Hence we can  and do restrict our attention to $G$ for which we stop growing $H''$ before it contains $C'_T \log n$ vertices.

We note that at this point, for every $v \in V-F_{i^*}-H-H''$  there  is a set $T_v$ of at least $\frac{n}{(\log n)^3}$ disjoint triples  in $F_{i^*}-H-H''$ which together with $v$ induce a $P_4$, in which $v$ is an endpoint. 

By the maximality of $H$,  for every $v \in F_i-H-H''$, none of the $P_4$s  induced by $v$ and a triple in $T_v$ extend to an  $P_6$ in $F_{i^*} \cup F_i-H$.
For each $i \neq i^*$,  for each component $K$ of $F_i-H-H''$ which has size at least 2. 
We apply  this fact to a vertex  $v$ in $K$ missing the rest of $K$.

If $|K| >20$, we greedily  select a family $P_v$   of $\lceil \frac{|K|}{20} \rceil$ pairs 
preferring pairs of adjacent vertices. For each  adjacent pair in $P_v$ and triple in $T_v$ there is a choice of edges between them which yields that together 
with $v$ they yield a $P_6$. If $P$ is a pair of nonadjacent vertices in $P_v$ then there is a stable set in $K$ containing all but  less than $\frac{|K|+1}{10}$
vertices. Hence $|K|<\frac{4n}{5(\alpha(T)-1)}$. Thus there are at least $\frac{n}{10(\alpha(T)-1)}$ vertices in $F_i-H-H''-K$ and this set which forms the neighbourhood of $v$ 
in $F_i-H-H''$ is also the common neighbourhood of $v$ and the vertices in $P$ in $F_i-H-H''$. By the choice of $H''$, for each triple  in $T_v$, one of these common neighbours 
sees none of the triple.  Hence again,  there is a choice of edges between $P$ and the triple  which yields that $F_i \cap F_{i^*}-H-H''$
contains  a $P_6$. Thus, the number of choices of edges between the union of the elements of $T_v$ and the union of the elements 
of the $P_v$ is at most $2^{6|T_v||P_v|}(1-\frac{1}{2^6})^{|T_v||P_v|}=  2^{6|T_v||P_v|- \Omega(\frac{|K|n)}{(\log n)^3})}$. 
If $|K|<20$ then again we have that there are at least $\frac{n}{10(\alpha(T)-1)}$ vertices in $F_i-H-H''-K$ and this set which forms the neighbourhood of $v$ 
in $F_i-H-H''$ is also the common neighbourhood of the vertices of $K$, and hence by the choice of $H'$, for any triple in $T_v$, one of these common neighbours 
is nonadjacent to all of the triple. 
Thus, again the choices for the edges between $K$ and $F_{i^*}-H-H''$ is $2^{|K|(|\pi_i{|^*}-H-H'|-\Omega(\frac{n}{(\log n)^3}))}$. 
On the other hand,  since $K-v$ is the disjoint union of a clique and a stable set, , the number of choices for $K$ and the edges from its vertices within $F_i-H-H''$  is at most ${n \choose |K|}2^{|K|}
=O(2^{(2|K|+1)\log n})$. 

It  follows that the subset of the $T$-free graphs extending the patterns we are considering for which $|H''|>(\log n)^5$ is a $o(1)$ proportion of the $T$-free graphs.
Hence we can and do restrict ourselves to patterns for which this is not the case. We note that  the restriction of $\pi$ to $G-H-H''$ is a $T$-freeness witnessing partition 
where for $i \neq i^*$, $\pi_i-H-H''$ is a clique. 

We set $t_i(v)=m\in((F_i \cap N(v), F_i-N(v))$ and call $v$ extreme on $\pi_i$ if for some sufficiently small $\psi>0$,  $t_i(v)< \psi n$. 

We consider first the possibility that for some $v$, $\min_{i=1}^{\alpha(T)-1} t_i(v)> \frac{n}{\log \log n}$. 
Our bound on the number of  vertices in singleton or large components   of $\overline {F_{i^*}-H}$, implies   $F_i-H$ contains 
 $\omega(\frac{n}{(\log n)^4})$ disjoint triples with  all of which  $v$ induces   $J$ which is either a $P_4$ or the disjoint union of a $P_3$ and a vertex. 
Further since for $i \neq i^*$, $F_i-H-H''$ is a clique we obtain that  $F_i$,contains 
 a family of  $\frac{n}{3 (\log \log n)}$ disjoint edges for  each of which which $v$ sees one endpoint and 
a family of  $\frac{n}{3 (\log \log n)}$ disjoint  edges for  each of which which $v$ sees no  endpoint.
 It follows that  for such a choice of edges, $v$ is
$n^{2/3}$-choice destroying  for the partition $v,F_1-H,F_2-H,...,F_{\alpha(T)-1}-H$. 
So by Lemma \ref{dangerous2cor} , the  number  of $T$-free graphs  extending the type of pattern we are considering  for which such a $v$ exists are  at most 
$2^n2^{-\omega(n^{4/3})}2^{(1-\frac{1}{\alpha(T)-1})\frac{n^2}{2}}$. Hence they are     an $o(1)$ proportion of the $T$-free graphs on $n$ vertices. Thus, we need only consider choices for which no such $v$ exists and every vertex is extreme on at least one $F_i-H_i$. 

We  let $DoublyExtreme$ be the subset of $H \cup H''$ which are extreme on $\pi_i$ for two $i$. 
For any $v$ added to DoublyExtreme, the number of choices for the edges from 
$v$ in $G$ is at most $2^{(1-\frac{2}{\alpha(T)-1})n+2\psi n}$. 

We partition $H -DoublyExtreme$ into $Extreme_1,...,Extreme_{\alpha(T)-1}$ where the vertices of $Extreme_i$ are extreme on $\pi_i$.
We consider the partition $ \pi'$  of $V-DoublyExtreme$ where $\pi'_i=(F_i-H) \cup Extreme_i$. 
We note that, by our bound on $k$, for $v \in  (F_i -H) \cup Extreme_i$, $t_i(v)<\frac{n}{\log \log n}$.

We recursively construct $H'''$ by adding in  (a)  sets of  four vertices  in  $\pi'_{i^*} -H'''$ which induce either 
a $P_4$ or a graph with 2 edges, or (b) sets of six vertices in $F_{i^*} \cup F_i-H$  for some $i \neq i^*$ which induce a $P_6$.
We note that the restriction of $\pi'_i$ to  $G-(H \cup H'' -Extreme_{i^*})-H'''-DoublyExtreme$ is a $T$-freeness certifying partition. 
We note that  the proportion of $T$-free graphs which are extensions of patterns we are considering for which  for any set $S$ of size 4 added to $H'''$, the choices of edges from $S$ make $S$ $n^{2/3}$-choice destroying for this partition  or for any set of size 6  added to $H'''$, the choices of edges from $S$ make $S$ $n^{2/3}$-choice destroying for the coarsening of the partition where we combine the two elements intersecting $S$  is  a $o(1)$ proportion of the $T$-free graphs. Hence we can and do assume there is no such $S$. 
Hence for each such $S$,  for some $\mu>0$, the number of choices for $S$ and  the edges from it is at most $n^{|S|}{n \choose n/ \log \log n}^{|S|}(\alpha(T)-1)2^{n(1-\frac{1}{\alpha(T)-1}-\mu)}$.
Hence the number of choices for $H'''$ and the edges from it is at most $2^{|H'''|n(1-\frac{1}{\alpha(T)-1}-\frac{\mu}{8})}$

We let $t^*=\max_{j\neq {i^*}} (\max _{v \in Extreme_j } t_j(v))$ and let $v^*$ and $j^*$ be such that $v^* \in Extreme_{j^*}$ and $t_{j^*}(v^*)=t^*$.
There are   $l \ge \frac{n}{(\log n)^5}$ triples in $F_{i^*}-H'''-DoublyExtreme$ with which $v$ induces a $P_4$ or the disjoint union of a $P_3$ and a vertex.

If $t^* > (\log n)^9$,  then there are at least $\frac{t^*}{4}$ edges of $F'_{j^*}-H'''-DoublyExtreme$ such that for each choice of a triple  and an edge $e$,
some choice of the edges between the triple and the edge yields a $P_6$ induced by the triple, the edge and $v^*$. By the maximality of $H'''$,
the proportion of choices for the edges between   $F'_{j^*}-H'''-DoublyExtreme$ and $F'_{i^*}-H'''-DoublyExtreme$ which extend such a choice of pattern   is at most  $(1-\frac{1}{2^6})^{t^*l}<e^{-\frac{nt^*}{(\log n)^6}}$.
On the other hand the number of choice of the edges  within the patterns from the vertices in the $Extreme_i$ for $i \neq j$ 
is ${n \choose |H  \cup H''|}{n \choose t^*}^{|H \cup H''|}=e^{O(t^*(\log n)^7}$. So, in this case we are done. 

If $t^*<(\log n)^9$ the number of choice of the edges  within the patterns from the vertices in the $Extreme_i$ for $i \neq j$ 
is $e^{o((\log n)^{16})}$. In this case, if  some $F'_i-H''-DoublyExtreme$ is not a clique then by the maximality of $H$ the  proportion of choices for the edges between   $F'_i-H'''-DoublyExtreme$ and $F'_{i^*}-H'''-DoublyExtreme$ which extend such a choice of pattern to a $T$-free graph  is at most  $(1-\frac{1}{2^6})^{l}<e^{-\frac{n}{(\log n)^6)}}$ and again we are done. So,
$\pi'$ is a $T$-freeness certifying partition of $G-H'''-DoublyExtreme$. Hence if $DoublyExtreme \cup H'''$ is nonempty, $G$ has a $T$-freeness certifying partition,
and there are no bad graphs extending it. Otherwise, we are done by our bounds on the number of edges leaving $H'''$ and $DoublyExtreme$. 
\end{proof}


\begin{thebibliography}{99}

\bibitem[ABBM11]{ABBM11} N. Alon, J. Balogh, B. Bollob\'{a}s and R. Morris.
\newblock The structure of almost all graphs in a hereditary property.
\newblock J. Combin. Theory B 101: 85-110 (2011).


\bibitem[BaBS11]{BaBS11}
\newblock J.~Balogh, B.~Bollob\'as, and M. Simonovits.
\newblock The fine structure of octahedron-free graphs,
\newblock {\em J. Combin. Theory Ser. B} {\bf 101(2)} (2011), 67--84.


\bibitem[BB11]{BB11} J. Balogh and J. Butterfield.
\newblock Excluding induced subgraphs: Critical graphs.
\newblock Random Structures and Algorithms, 38(1-2): 100-120 (2011).

\bibitem[BT10]{BT10} D. Berend and T. Tassa. 
\newblock Improved bounds on Bell numbers and on moments of sums of random variables.
\newblock Probability and Mathematical Statistics 30(2): 185-205 (2010).

\bibitem[BT97]{BT97} B. Bollob\'{a}s and A. Thomason,
\newblock Hereditary and monotone properties of graphs.
\newblock The mathematics of Paul Erd\H{o}s II 14: 70-78 (1997).

\bibitem[BMbook]{BMbook} J.-A. Bondy and U.S.R Murty.
\newblock Graph theory. 
\newblock Graduate texts in mathematics, Springer, (2007). 

\bibitem[B58]{B58} N. G. de Bruijn.
\newblock Asymptotic Methods in Analysis.
\newblock Dover, New York, NY: 103-108 (1958).

\bibitem[C89]{C89} A. Cayley.
\newblock A theorem on trees.
\newblock Quart. J. Pure Appl. Math. 23: 376-378 (1889); Collected Mathematical Papers Vol. 13, Cambridge University Press: 26-28, (1897).

\bibitem[C81]{C81} H. Chernoff.
\newblock A note on an inequality involving the normal distribution.
\newblock Ann. Probab., 9: 533-535 (1981).

\bibitem[EKR76]{EKR76}
\newblock P. Erd\H os, D.~J.~Kleitman, B.~L.~Rothschild.
\newblock Asymptotic enumeration of $K_n$-free graphs,
\newblock {\em International Colloquium on Combinatorial Theory, Atti dei Convegni Lincei} {\bf 17} (1976), 19--27.

\bibitem[ES35]{ESz35} P. Erd\H{o}s and G. Szekeres. 
\newblock A combinatorial problem in geometry. 
\newblock Compositio Math. 2:  463-470 (1935).

\bibitem[KKOT15]{KKOT15} J. Kim, D. K\"{u}hn, D. Osthus, T. Townsend.
\newblock Forbidding induced even cycles in a graph: typical structure and counting.
\newblock J. Comb. Theory, Ser. B {\bf 131} 170-219 (2018).

\bibitem[K97]{K97} D.E. Knuth. 
\newblock The Art of Computer Programming, Volume 3, 2nd ed.
\newblock Addison-Wesley: 73-75 (1997).

\bibitem[KPR87]{KPR87} Ph.G. Kolaitis, H.J. Pr{\"o}mel and B.L. Rothschild.
\newblock $K_{\ell+1}$-free graphs: asymptotic structure and a 0 - 1 law.
\newblock Trans. Amer. Math. Soc. 303: 637-671 (1987).

\bibitem[PRY18]{PRY18} J. Pach, B. Reed and Y. Yuditsky.
\newblock Almost all string graphs are intersection graphs of plane convex sets.
\newblock Proceedings of the 34th International Symposium on Computational Geometry (SoCG), (2018).

\bibitem[PS91]{PS91} H. J. Pr{\"o}mel and A. Steger.
\newblock Excluding induced subgraphs: Quadrilaterals.
\newblock Random Structures and Algorithms 2: 55-71 (1991).

\bibitem[PS92]{PS92} H. J. Pr{\"o}mel and A. Steger.
\newblock Excluding induced subgraphs III: A general asymptotic.
\newblock Random Structures and Algorithms 3(1): 19-31 (1992).

\bibitem[PSBerge]{PSBerge} H. J. Pr{\"o}mel and A. Steger.
\newblock Almost all Berge graphs are perfect.
\newblock Combinatorics, Probability and Computing 1: 53-79 (1992).

\bibitem[PS93]{PS93} H. J. Pr{\"o}mel and A. Steger.
\newblock Excluding induced subgraphs II: Extremal graphs.
\newblock Discrete Applied Mathematics 44(1-3): 283-294 (1993).

\bibitem[R30]{R30} F.P. Ramsey.
\newblock On a Problem of Formal Logic. 
\newblock Proceedings of the London Mathematical Society, s2-30: 264-286 (1930).

\bibitem[R]{R} B. Reed. 
\newblock The Global Structure of a Typical Graph without $H$ as an induced subgraph when $H$ is a cycle.
\newblock To be submitted. 

\bibitem[RY]{RY} B. Reed and Y. Yuditsky. 
\newblock The Asymptotic $\chi$-Boundedness of Hereditary Families
\newblock To be submitted. 

\bibitem[SZ94]{SZ94} E. Scheinerman and J. Zito.
\newblock On the size of hereditary classes of graphs.
\newblock J Combin Theory Ser B 61: 16-39 (1994).

\bibitem[S74]{S74} D. Seinsche
\newblock On a property of the class of $n$-colorable graphs.
\newblock J. Combin. Theory Ser. B 16: 191-193 (1974).


\end{thebibliography}
\end{document}